\documentclass[10pt]{amsart}
\usepackage{amsmath}
\usepackage{amsthm}
\usepackage{color}
\usepackage{amscd}
\usepackage{amsfonts}
\usepackage{graphicx}
\usepackage{epsfig}
\usepackage{amssymb,latexsym}

\newcommand{\dv}{\operatorname{div}}

\newcommand{\sgn}{\operatorname{sign}}

\def\bu{\mathbf{u}}

\def\bv{\mathbf{v}}

\def\bx{\mathbf{x}}

\def\R{\mathbf{R}}

\def\F{\mathbf F}

\def\A{\mathcal A}

\def\bA{\mathbf A}

\def\bcero{\mathbf 0}

\newtheorem{theorem}{Theorem}[section]

\newtheorem{lemma}[theorem]{Lemma} 
\newtheorem{proposition}[theorem]{Proposition}

\begin{document}

\title{On a variant of Tykhonov regularization in optimal control under PDEs}
\author{Pablo Pedregal}
\date{} 
\thanks{INEI, U. de Castilla-La Mancha, 13071 Ciudad Real, SPAIN. Supported by grant 
MTM2017-83740-P}
\begin{abstract}
We make some remarks on a variant of the classical Tikhonov regularization in optimal control under PDEs which allows for a certain flexibility in dealing with non-linearities and state restrictions,  in the sense that differential constraints between control and state are eliminated and pairs can run freely in their respective sets of feasibility, at the expense of introducing an additional variable in a collection of approximated problems. In addition to exploring basic issues like existence and optimality, we also discuss a numerical procedure and apply it to some academic, illustrative numerical tests, as well as examine the convergence of solutions of this new family of approximated problems to the solutions of the underlying optimal control problem. 
\end{abstract}
\maketitle
\section{Introduction}
In formal terms, optimal control focuses on situations in which one would like to minimize a certain cost functional 
$$
I[\bu, \bv]:\A\to\R
$$
where both sets of variables $(\bu, \bv)$ are typically coupled through a (non-local) differential law
$$
\F(\bu, \bv)=\bcero.
$$
$\A$ is the set of competing pairs understood as a certain subset of an appropriate functional space, and $\F$ stands for the differential operator, together with appropriate boundary conditions, that produces, in a unique way, the output $\bu$, the state variable, once the input $\bv$, the control variable, has been provided. These optimization problems are very well understood in various contexts and under varying sets of assumptions, and their range of applicability knows no limits. There are two main sources of concern that push the analysis to more complicated frameworks:
\begin{enumerate}
\item if the operator $\F(\bu, \bv)=\bcero$ is non-linear in $\bu$ for given $\bv$, then finding the state $\bu$, for a prescribed control law $\bv$, may be quite involved depending on the nature of the non-linearity, and, in practice, may require an iterative mechanism to find, or approximate, $\bu$ once $\bv$ is known;
\item the situation in which additional, typically pointwise,  constraints, for both state $\bu$ and control $\bv$, are to be enforced, is even more dramatic because a practical understanding of the differential law $\F(\bu, \bv)=\bcero$ related to pointwise constraints seems to be hopeless.
\end{enumerate}
See \cite{casasmateos}, \cite{nst}, \cite{troltzsch} and more especialized literature therein to appreciate the difficulties associated with state constraints. 

We would like to stress in this contribution how, by allowing a bit of flexibility in the formulation of the problem, both situations can be eased and the analysis carried out in a more affordable and straightforward manner. Our motivation, as stated in the abstract, is to free problems from having to deal with difficult constraints at the expense of introducing new variables and approximations. Though many situations can be treated, we will stick to two typical, explicit scenarios to convey our remarks. 

The first one deals with a semilinear elliptic equation as state law. Our model problem involves a standard Tikhonov regularization (\cite{casasmateos}) of the type
$$
\hbox{Minimize in }v\in L^2(\Omega):\quad 
I[v]=\int_\Omega\left(\frac12|u(\bx)-\overline u(\bx)|^2+\frac\mu2|v(\bx)|^2\right)\,d\bx
$$
where
\begin{equation}\label{estado}
-\dv(\nabla u)+\phi(u)=v\hbox{ in }\Omega,\quad u=0\hbox{ in }\partial\Omega.
\end{equation}
Here $\Omega\subset\R^N$ is a regular, bounded domain, $\overline u$ is a certain desired target function in $L^2(\Omega)$, $\mu>0$, and $\phi$ is a real function with the appropriate growth at infinity so that the composition $\phi(u)$ belongs to $L^2(\Omega)$ for every $u\in H^1_0(\Omega)$. The whole point is that the state law \eqref{estado} requires to solve a non-linear PDE to find the state $u$ once the control $v$ has been specified. In particular, those two variables $u$ and $v$ are strongly coupled and cannot be specified independently of each other in $H^1_0(\Omega)\times L^2(\Omega)$. 

The second one addresses the issue of pointwise constraints for state $u$ and control $v$. To concentrate specifically on this issue, we choose a linear state equation to simplify other ingredients, and deal with the problem
$$
\hbox{Minimize in }(u, v)\in H^1_0(\Omega)\times L^\infty(\Omega):\quad 
I[u, v]=\int_\Omega\frac12|u(\bx)-\overline u(\bx)|^2\,d\bx
$$
where
\begin{gather}
-\dv(\nabla u)+u=v\hbox{ in }\Omega,\quad u=0\hbox{ in }\partial\Omega,\nonumber\\
u(\bx)\le 0,\quad v_-(\bx)\le v(\bx)\le v_+(\bx)\quad\hbox{ for a.e. }\bx\in\Omega,\nonumber
\end{gather}
where $v_\pm$ are prescribed $L^\infty(\Omega)$-functions. It is virtually impossible to anticipate the conditions on the control $v$ that ensure that the corresponding state $u$ will comply with the condition $u\le0$, and this is the main difficulty with state constraints. 

As already indicated, in both situations, we would like to enable a bit of flexibility in the state equation in such a way that both variables $u$ and $v$ can be given freely and independently of each other in their respective sets of feasibility, either $H^1_0(\Omega)\times L^2(\Omega)$, or
$$
\{u\in H^1_0(\Omega): v\le0\}\times\{v\in L^\infty(\Omega): v_-\le v\le v_+\}.
$$
This flexibility can be implemented in several ways but the possibility that we would like to focus on amounts to introducing a ``defect" or residual function $w\in H^1_0(\Omega)$ that is determined, in a unique way once $u$ and $v$ are given, through the equation
$$
-\dv(\nabla u+\nabla w)+\phi(u)=v\hbox{ in }\Omega,
$$
in the first situation, or 
$$
-\dv(\nabla u+\nabla w)+u=v\hbox{ in }\Omega,
$$
in the second. In addition, to account for a small size of this defect $w$ and not stay too far from the initial optimal control problemR, we change the cost functional to
$$
I[u, v]=\int_\Omega\left(\frac12|u(\bx)-\overline u(\bx)|^2+\frac\mu2|v(\bx)|^2+\frac\lambda2|\nabla w(\bx)|^2\right)\,d\bx
$$
or
$$
I[u, v]=\int_\Omega\left(\frac12|u(\bx)-\overline u(\bx)|^2+\frac\lambda2|\nabla w(\bx)|^2\right)\,d\bx,
$$
respectively. If the parameter $\lambda$ is large, a small size of the defect $w$ is expected and so we will be rather close to the true, exact state law. 

We will therefore examine the two problems
\begin{equation}\label{primero}
\hbox{Minimize in }(u, v):\quad
I[u, v]=\int_\Omega\left(\frac12|u(\bx)-\overline u(\bx)|^2+\frac\mu2|v(\bx)|^2+\frac\lambda2|\nabla w(\bx)|^2\right)\,d\bx
\end{equation}
under
\begin{gather}
(u, v)\in H^1_0(\Omega)\times L^2(\Omega),\nonumber\\
-\dv(\nabla u+\nabla w)+\phi(u)=v\hbox{ in }\Omega,\quad w=0\hbox{ on }\partial\Omega,\nonumber
\end{gather}
and
\begin{equation}\label{segundo}
\hbox{Minimize in }(u, v):\quad
I[u, v]=\int_\Omega\left(\frac12|u(\bx)-\overline u(\bx)|^2+\frac\lambda2|\nabla w(\bx)|^2\right)\,d\bx
\end{equation}
under
\begin{gather}
(u, v)\in \A=\{(u, v)\in H^1_0(\Omega)\times L^2(\Omega): u\le0, v_-\le v\le v_+\},\nonumber\\
-\dv(\nabla u+\nabla w)+u=v\hbox{ in }\Omega,\quad w=0\hbox{ on }\partial\Omega.\nonumber
\end{gather}

A much more general framework is allowed by taking
$$
\hbox{Minimize in }(u, v)\in\A:\quad \int_\Omega L(\bx, u(\bx), v(\bx), \nabla w(\bx))\,d\bx
$$
subject to
$$
-\dv[\bA(\bx, \nabla u(\bx))+\nabla w(\bx)]+\phi(\bx, u(\bx))=\psi(\bx, v(\bx))\hbox{ in }\Omega,\quad
w=0\hbox{ on }\partial\Omega,
$$
and we typically throw pointwise, or other form of (non-differential) conditions into the admissible set $\A$. $L$, $\phi$, $\psi$ and $\bA$ need to comply with appropriate hypotheses which we do not bother to specify at this stage, as we will deal directly with the two indicated problems to convey a few remarks. We will treat successively existence of optimal solutions, optimality conditions, and direct, practical numerical approximation. These approximations are given for the sake of illustration, as finer experiments would require an expertise that the author cannot claim to posses. 
Finally, we will make a few observations concerning the limit behavior as the parameter $\lambda$ is taken to $+\infty$. 

Optimal control governed by PDEs have been examined thoroughly since the pioneering work of J. L. Lions (see \cite{lions1}, \cite{lions}, \cite{lions2} for instance) because of the mathematical methods these problems require, and their uncountable applications in all fields of Science and Engineering. Other important references, that also deal with the numerical approximation and in addition to those already mentioned above without any attempt to be exhaustive, are \cite{gugat}, \cite{kogutleugering}, \cite{lastriuno}, \cite{lastridos}, \cite{volkwein}. 

\section{Main results for the unconstrained situation}\label{dos}
As indicated, we will stick, for the sake of definiteness, to the model problem
$$
\hbox{Minimize in }(u, v):\quad \int_\Omega\left(\frac12|u(\bx)-\overline u(\bx)|^2+\frac\mu2|v(\bx)|^2+\frac{\lambda}2|\nabla w(\bx)|^2\right)\,d\bx
$$
subject to $(u, v)\in H^1_0(\Omega)\times L^2(\Omega)$, and 
\begin{equation}\label{state}
-\dv[\nabla u(\bx)+\nabla w(\bx)]+\phi(u(\bx))=v(\bx)\hbox{ in }\Omega,\quad
w=0\hbox{ on }\partial\Omega.
\end{equation}
Let us stress again that the main reason to introduce the additional variable $w$ into the problem is to let pairs $(u, v)$ run freely in the feasible space $H^1_0(\Omega)\times L^2(\Omega)$ so that there is no dependence of $u$ on $v$.

Our main assumptions are:
\begin{enumerate}
\item $\Omega\subset\R^N$ is a regular, bounded domain (as regular as we may need it to be);
\item the function $\overline u\in L^2(\Omega)$;
\item $\mu, \lambda>0$;
\item $\phi:\R\to\R$ is a smooth function and has, at most, growth $N/(N-2)$ at infinity so that the composition $\phi(u)\in L^2(\Omega)$ for every $u\in H^1_0(\Omega)$.
\end{enumerate}

Under these assumptions, equation \eqref{state} is always well-defined, and so there is a unique solution $w\in H^1_0(\Omega)$ for every admissible pair $(u, v)\in H^1_0(\Omega)\times L^2(\Omega)$. 

On the other side of the spectrum, we can consider a general, non-quadratic cost under a linear state equation, like
\begin{equation}\label{statee}
\hbox{Minimize in }(u, v):\quad \int_\Omega\left(\psi(u(\bx), v(\bx), \bx)+\frac{\lambda}2|\nabla w(\bx)|^2\right)\,d\bx
\end{equation}
subject to $(u, v)\in H^1_0(\Omega)\times L^2(\Omega)$, and 
$$
-\dv[\nabla u(\bx)+\nabla w(\bx)]=v(\bx)\hbox{ in }\Omega,\quad
w=0\hbox{ on }\partial\Omega,
$$
or any other linear PDE tying $w$ to the pair $(u, v)$. We will be using later this particular version of a unconstrained problem for the direct numerical approximation under state and control pointwise constraints. 

The treatment for both versions of a unconstrained problem, or the more general version incorporating non-linearities both in the state equation and the functional cost for that matter, is formally the same so that, for the sake of brevity, we will focus on the first one for our treatment of existence, optimality, and numerical approximation. 

\subsection{Existence of optimal solutions}
The first important issue is that of existence of optimal solutions. 
\begin{theorem}\label{existencia}
Suppose, in addition to the hypotheses indicated above, that there is an affine function $l(u)$ such that $(l(u)-\phi(u))u\le M$ for all $u\in\R$, and some non-negative constant $M$. Then
there are optimal pairs $(u, v)\in H^1_0(\Omega)\times L^2(\Omega)$ for our optimal control problem. If, in addition, $\phi$ is affine, then the optimal pair is unique.
\end{theorem}
\begin{proof}
Note that the functional
$$
I[u, v]=\int_\Omega\left(\frac12|u(\bx)-\overline u(\bx)|^2+\frac\mu2|v(\bx)|^2+\frac{\lambda}2|\nabla w(\bx)|^2\right)\,d\bx
$$
is non-negative. Let $\{(u_j, v_j)\}\subset H^1_0(\Omega)\times L^2(\Omega)$ be a minimizing sequence, and let $w_j\in H^1_0(\Omega)$ be such that
\begin{equation}\label{estadoo}
-\dv[\nabla u_j(\bx)+\nabla w_j(\bx)]+\phi(u_j(\bx))=v_j(\bx)\hbox{ in }\Omega.
\end{equation}
We immediately deduce that $\{u_j\}$, $\{v_j\}$, and $\{w_j\}$ are uniformly bounded in $L^2(\Omega)$, in $L^2(\Omega)$, and in $H^1_0(\Omega)$, respectively, and so, for a subsequence which we do not care to relabel, we can assume that
$$
u_j\rightharpoonup u\hbox{ in }L^2(\Omega),\quad
v_j\rightharpoonup v \hbox{ in }L^2(\Omega),
\quad w_j\rightharpoonup w\hbox{ in }H^1_0(\Omega).
$$
The main step of the proof focuses on showing that in fact $u_j\rightharpoonup u$ in $H^1_0(\Omega)$. Indeed, from \eqref{estadoo}, using $u_j$ itself as a test function, we find that
\begin{equation}\label{igualdad}
\int_\Omega[|\nabla u_j|^2+\nabla w_j\cdot\nabla u_j+(\phi(u_j)-l(u_j))u_j]\,d\bx=
\int_\Omega  (v_j-l(u_j))u_j\,d\bx.
\end{equation}
The uniform-boundedness condition assumed on the product $(l(u_j)-\phi(u_j))u_j$ enables us to write, for some constant $C$, 
$$
 \|\nabla u_j\|_{L^2(\Omega)}^2\le \|w_j\|_{H^1_0(\Omega)}\|\nabla u_j\|_{L^2(\Omega)}+\|u_j\|_{L^2(\Omega)}\|v_j\|_{L^2(\Omega)}+M|\Omega|+C\|u_j\|_{L^2(\Omega)}^2.
 $$
 Since the last three terms are uniformly bounded for all $j$, as well as the norm of $w_j$, this condition is a quadratic inequality for the non-negative number $\|\nabla u_j\|_{L^2(\Omega)}$ with uniformly bounded coefficients of the form
 $$
 x^2\le ax+b,\quad x=\|\nabla u_j\|_{L^2(\Omega)}
 $$
 and $a$ and $b$ fixed, positive numbers. Thus
the sequence of numbers $\{\|\nabla u_j\|\}$ is also uniformly bounded. 

We can, therefore, assume, without loss of generality, that the weak convergence $u_j\rightharpoonup u$ takes place in $H^1_0(\Omega)$. In this case, for a test function $\theta(\bx)\in H^1_0(\Omega)$, we will have
$$
\int_\Omega(\nabla u_j\cdot\nabla\theta+\nabla w_j\cdot\nabla\theta+\phi(u_j)\theta)\,d\bx=\int_\Omega v_j\theta\,d\bx.
$$
By taking limits in $j$ in both sides of this identity, bearing in mind the strong convergence $u_j\to u$ in $L^2(\Omega)$, we are led to
$$
\int_\Omega(\nabla u\cdot\nabla\theta+\nabla w\cdot\nabla\theta+\phi(u)\theta)\,d\bx=\int_\Omega v\theta\,d\bx,
$$
and so $w$ is the corresponding ``defect" for the pair $(u, v)\in H^1_0(\Omega)\times L^2(\Omega)$. Because the functional is convex in $u$, $v$ and $w$, we have that
$$
I[u, v]\le\liminf_{j\to\infty}I[u_j, v_j],
$$
and the pair $(u, v)$ is optimal for the problem. 

The uniqueness, in case $\phi$ is affine, is a direct consequence of the strict convexity of the functional. This is standard. 
\end{proof}

\subsection{Optimality}
To write optimality conditions we use capital letters to indicate feasible variations of the various functions involved. In this way, if we formally write to first-order in $\epsilon$
\begin{equation}\label{perture}
-\dv[\nabla u+\epsilon\nabla U+\nabla w+\epsilon\nabla W]+\phi(u+\epsilon U)=v+\epsilon V\hbox{ in }\Omega
\end{equation}
 for a local perturbation $W$ produced in $w$ due to the perturbations $U$ and $V$ on $u$ and $v$, respectively, and differentiate with respect to $\epsilon$, setting $\epsilon=0$ afterwards, we find
\begin{equation}\label{pertur}
-\dv[\nabla U+\nabla W]+\phi'(u)U=V\hbox{ in }\Omega,
\end{equation}
where $U, W\in H^1_0(\Omega)$ while $V\in L^2(\Omega)$. Equation \eqref{pertur} is regarded as the equation providing the perturbation $W$ in $w$ once changes $U$ on $u$ and $V$ on $v$ have been given.

Concerning the functional, we would have to differentiate the expression
$$
\int_\Omega\left(\frac12|u+\epsilon U-\overline u|^2+\frac\mu2|v+\epsilon V|^2+\frac{\lambda}2|\nabla w+\epsilon\nabla W|^2\right)\,d\bx
$$
with respect to $\epsilon$, and set $\epsilon=0$. It yields
\begin{align}
\langle I'[u, v], (U, V)\rangle=&\left.\frac{d}{d\epsilon}I((u, v)+\epsilon(U, V))\right|_{\epsilon=0}\nonumber\\
=&\int_\Omega[(u-\overline u)U+\mu vV+\lambda\nabla w\cdot\nabla W]\,d\bx\nonumber
\end{align}
where
\begin{gather}
w:\quad -\dv[\nabla u+\nabla w]+\phi(u)= v\hbox{ in }\Omega,\nonumber\\
W:\quad -\dv[\nabla U+\nabla W]+\phi'(u)U=V\hbox{ in }\Omega.\nonumber
\end{gather}
If we use $w$ as a test function in this second equation, then
$$
\int_\Omega\nabla w\cdot\nabla W\,d\bx=\int_\Omega [Vw-\phi'(u)Uw-\nabla U\cdot\nabla w]\,d\bx,
$$
and, taking this identity back to the derivative of the functional, we arrive at
$$
\langle I'[u, v], (U, V)\rangle=
\int_\Omega[(u-\overline u)U+\mu vV+\lambda(Vw-\phi'(u)Uw-\nabla U\cdot\nabla w)]\,d\bx.
$$
If the pair $(u, v)\in H^1_0(\Omega)\times L^2(\Omega)$ is indeed the optimal solution of the control problem, then the previous integral must vanish for arbitrary $U\in H^1_0(\Omega)$ and $V\in L^2(\Omega)$. This fact immediately leads to the following statement. 

\begin{theorem}
Under the same hypotheses as before, let $(u, v)\in H^1_0(\Omega)\times L^2(\Omega)$ be an optimal pair of the problem according to Theorem \ref{existencia}, and set $w\in H^1_0(\Omega)$ for the corresponding defect so that
\begin{equation}\label{deebil}
-\dv[\nabla u(\bx)+\nabla w(\bx)]+\phi(u(\bx))=v(\bx)\hbox{ in }\Omega.
\end{equation}
Then
\begin{gather}
u-\lambda\phi'(u)w=\overline u-\lambda\Delta w\hbox{ in }\Omega,\nonumber\\
\mu v+\lambda w=0\hbox{ in }\Omega.\nonumber
\end{gather}
\end{theorem}
The converse can hold in general only if the state equation is linear. If $\phi$ is linear (affine), and if the pair $(u, v)\in H^1_0(\Omega)\times L^2(\Omega)$ is such that for the unique solution $w$ of \eqref{deebil}, the two previous relations in the statement hold, then $(u, v)$ is the unique solution of the corresponding optimal control problem.

\section{The numerical simulation}
The previous computations can be put directly into an iterative descent approximation strategy. We describe this for the sake of illustration, not pretending to go into a rigorous discussion.
If we go back to the expression
$$
\left.\frac{d}{d\epsilon}I[(u, v)+\epsilon(U, V)]\right|_{\epsilon=0}=
\int_\Omega[(u-\overline u)U+\mu vV+\lambda(Vw-\phi'(u)Uw-\nabla U\cdot\nabla w)]\,d\bx,
$$
to find the steepest descent directions $U$ and $V$, we need to solve the two variational problems
$$
\hbox{Minimize in }U\in H^1_0(\Omega):\quad
\int_\Omega\left[\frac12|\nabla U|^2+(u-\overline u)U-\lambda\phi'(u)wU-\lambda\nabla w\cdot\nabla U\right]\,d\bx,
$$
and
$$
\hbox{Minimize in }V\in L^2(\Omega):\quad
\int_\Omega\left[\frac12|V|^2+(\mu v+\lambda w)V\right]\,d\bx.
$$
Their respective solutions are easily found to be
$$
-\dv[\nabla U-\lambda\nabla w]+(u-\overline u)-\lambda\phi'(u)w=0\hbox{ in }\Omega,\quad U=0\hbox{ on }\partial\Omega,
$$
and
$$
V+\mu v+\lambda w=0\hbox{ in }\Omega.
$$
We finally decide on the step size to be taken. To this aim, suppose we choose the direction $(U, V)$, solutions of these last two problems, and then try to determine the value of $\epsilon$ minimizing the function
\begin{align}
g(\epsilon)=&I[(u, v)+\epsilon(U, V)]\nonumber\\
=&\int_\Omega\left[\frac12|u+\epsilon U-\overline u|^2+
\frac\mu2|v+\epsilon V|^2+\frac\lambda2|\nabla w+\epsilon\nabla W|^2\right]\,d\bx.\nonumber
\end{align}
It is elementary to have
$$
\epsilon=-\frac{\int_\Omega\left[(u-\overline u)U+\mu vV+\lambda\nabla w\cdot\nabla W\right]\,d\bx}{\int_\Omega(U^2+\mu V^2+\lambda|\nabla W|^2)\,d\bx}.
$$
Note, however, that this formula furnishes just an approximation since the function $g(\epsilon)$ is not quadratic because the perturbation $W$ would also depend on $\epsilon$ when $\epsilon>0$. Note that equation \eqref{perture} is written with $W$ independent of $\epsilon$ as a first-order approximation in $\epsilon$.
In this way, we implement an iterative procedure in several steps:
\begin{enumerate}
\item Initialization. Take $u_0=v_0=0$, for instance.
\item Main iterative step until convergence. Suppose we know $(u_j, v_j)\in H^1_0(\Omega)\times L^2(\Omega)$. 
\begin{enumerate}
\item Solve successively for
\begin{gather}
w_j:\quad -\dv(\nabla u_j+\nabla w_j)+\phi(u_j)=v_j\hbox{ in }\Omega,\quad w_j=0\hbox{ on }\partial\Omega;\nonumber\\
U_j:\quad -\dv(\nabla U_j-\lambda\nabla w_j)+u_j-\overline u-\lambda \phi'(u_j)w_j=0\hbox{ in }\Omega,\quad U_j=0\hbox{ on }\partial\Omega;\nonumber\\
V_j:\quad V_j=-\mu v_j-\lambda w_j\hbox{ in }\Omega\nonumber.
\end{gather}
\item Stopping criterium. If the norm
$$
\|(U_j, V_j)\|^2=\int_\Omega\left(|\nabla U_j|^2+V^2_j\right)\,d\bx
$$
is sufficiently small, stop and take $(u_j, v_j)$ as a good approximation of the optimal pair. If not, proceed.
\item Solve the additional problem
$$
W_j:\quad -\dv(\nabla U_j+\nabla W_j)+\phi'(u_j)U_j=V_j\hbox{ in }\Omega,\quad W_j=0\hbox{ on }\partial\Omega,
$$
and compute the number
$$
\epsilon_j=-\frac{\int_\Omega\left[(u_j-\overline u)U_j+\mu v_jV_j+\lambda\nabla w_j\cdot\nabla W_j\right]\,d\bx}{\int_\Omega(U_j^2+\mu V_j^2+\lambda|\nabla W_j|^2)\,d\bx}.
$$ 
Note that this number is positive, because the function $g(\epsilon)$ above, if we regard $W$ as independent of $\epsilon$ as has been calculated in this step, is a convex, quadratic function with a negative derivative for $\epsilon=0$ if the pair $(U_j, V_j)$ is a descent direction.
\item Update rule. Set
$$
u_{j+1}=u_j+\epsilon_j U_j,\quad v_{j+1}=v_j+\epsilon_j V_j,
$$
and proceed iteratively. 
\end{enumerate}
\end{enumerate}

\begin{figure}[b]
\includegraphics[scale=0.3]{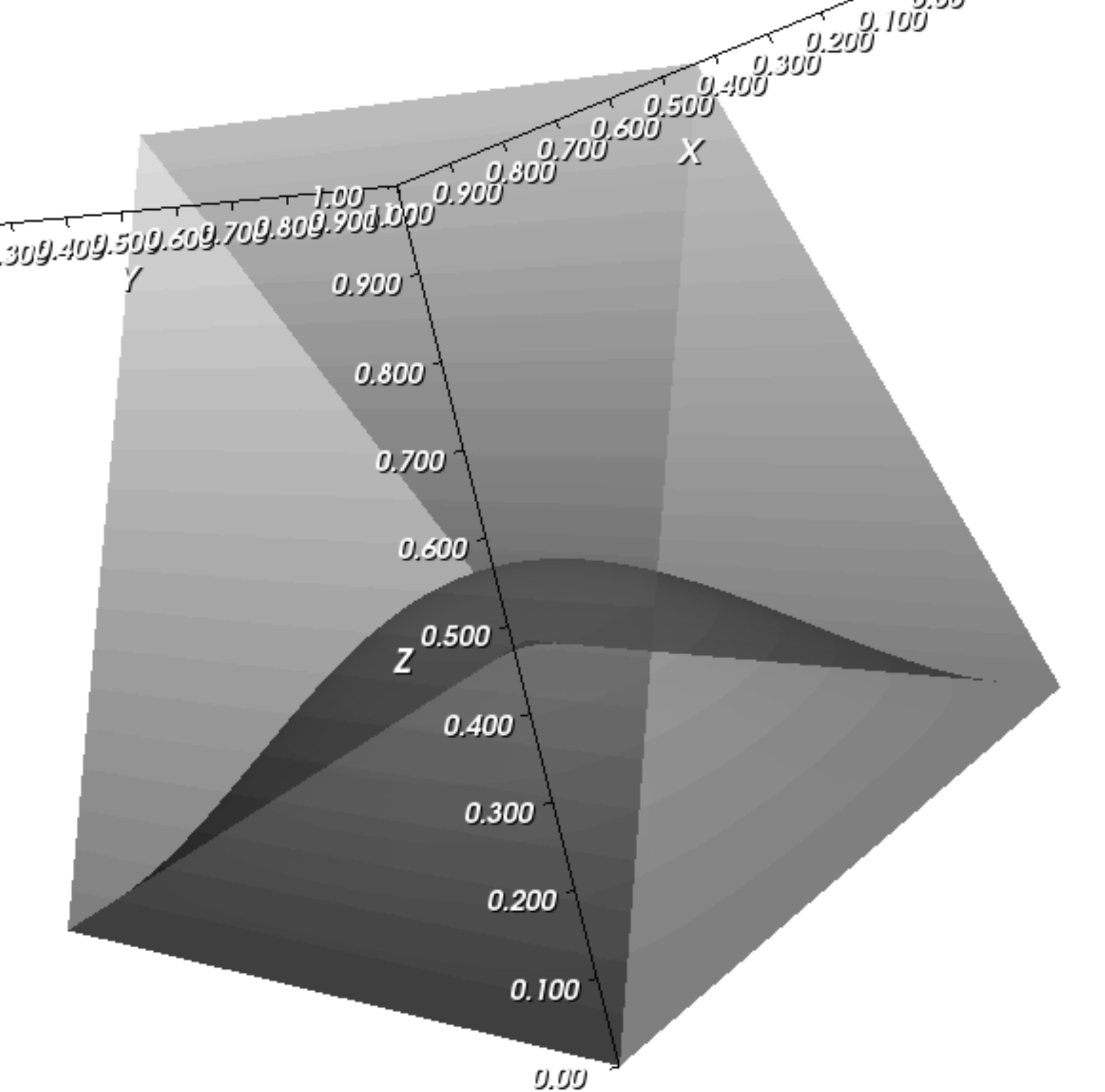}\quad
\includegraphics[scale=0.3]{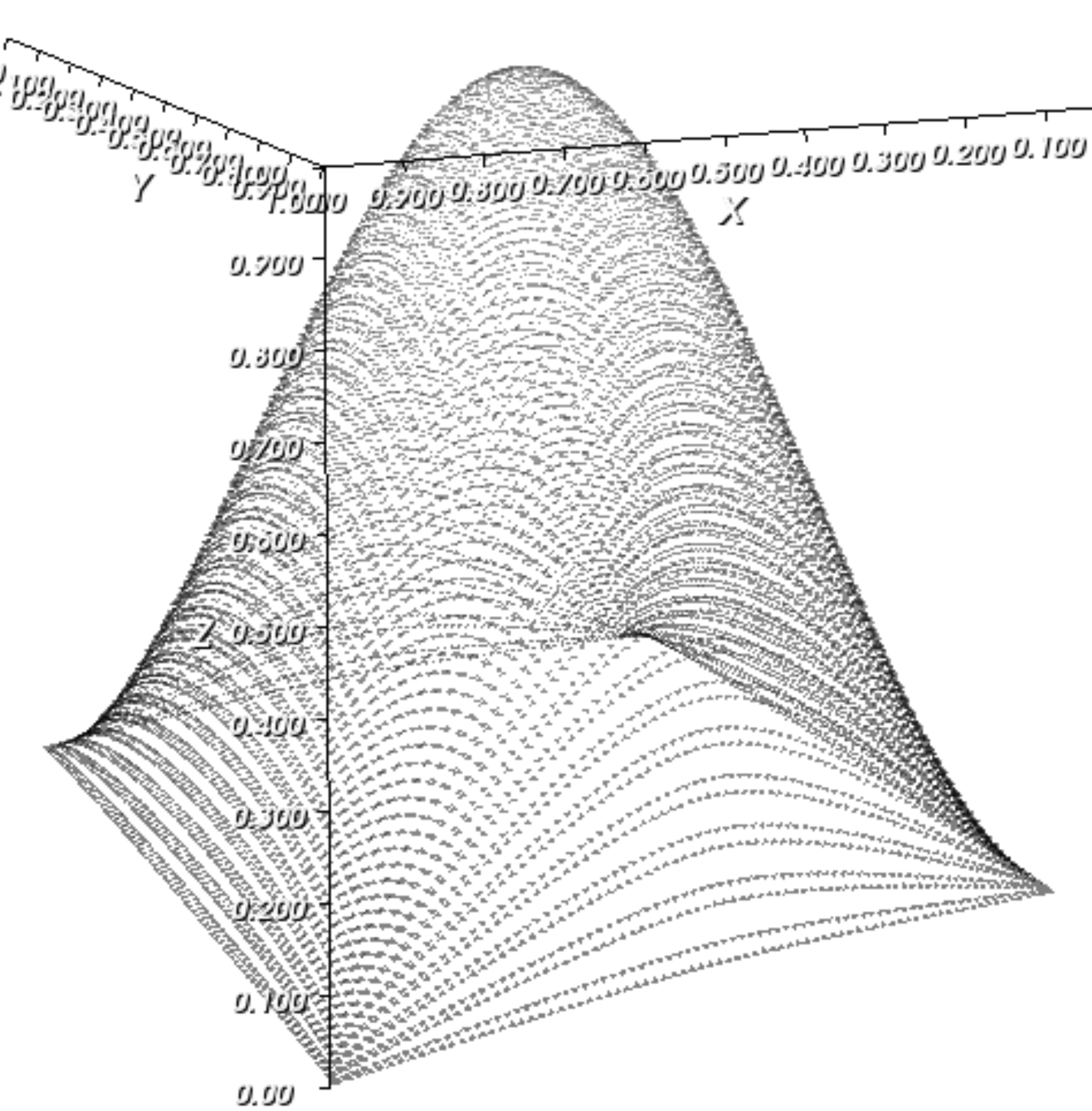}
\caption{A comparison between the target and the optimal profile (left) for $\mu=0.01$. The control $v$ on the right.}
\label{uno}       
\end{figure}

\begin{figure}[b]
\includegraphics[scale=0.29]{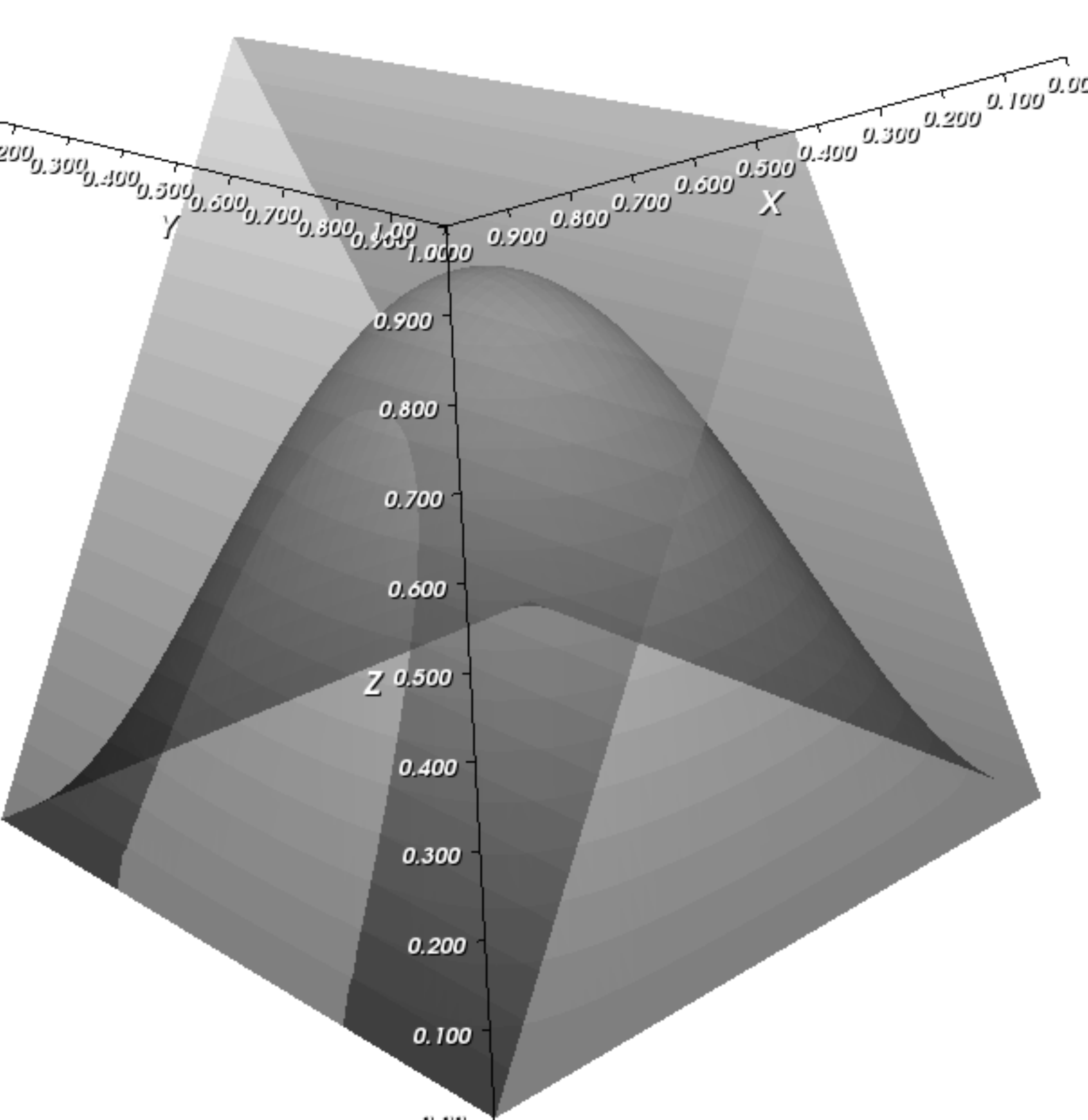}\quad
\includegraphics[scale=0.29]{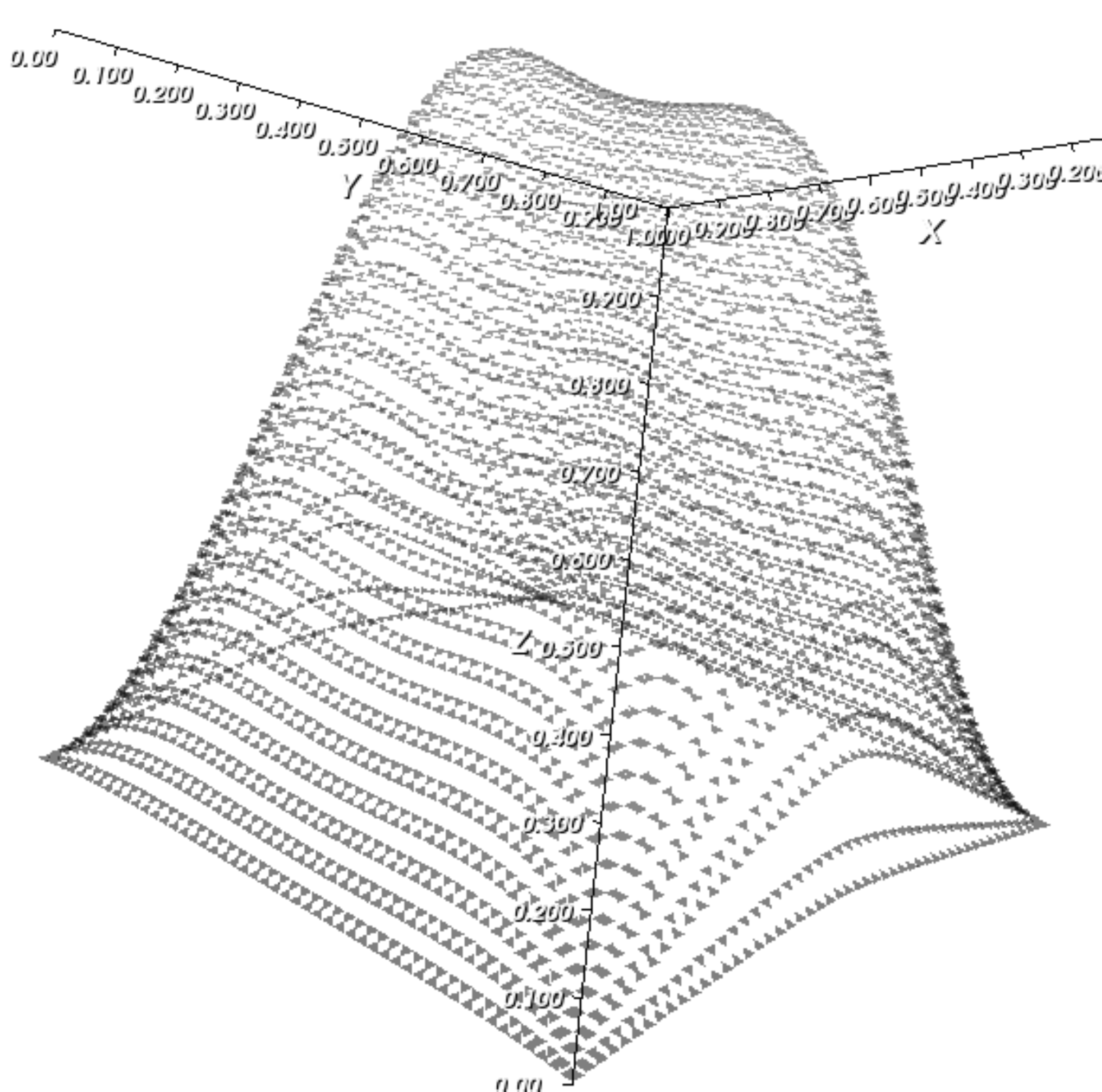}
\caption{A comparison between the target and the optimal profile (left) for $\mu=0.001$. The control $v$ on the right.}
\label{dos}       
\end{figure}

\begin{figure}[b]
\includegraphics[scale=0.29]{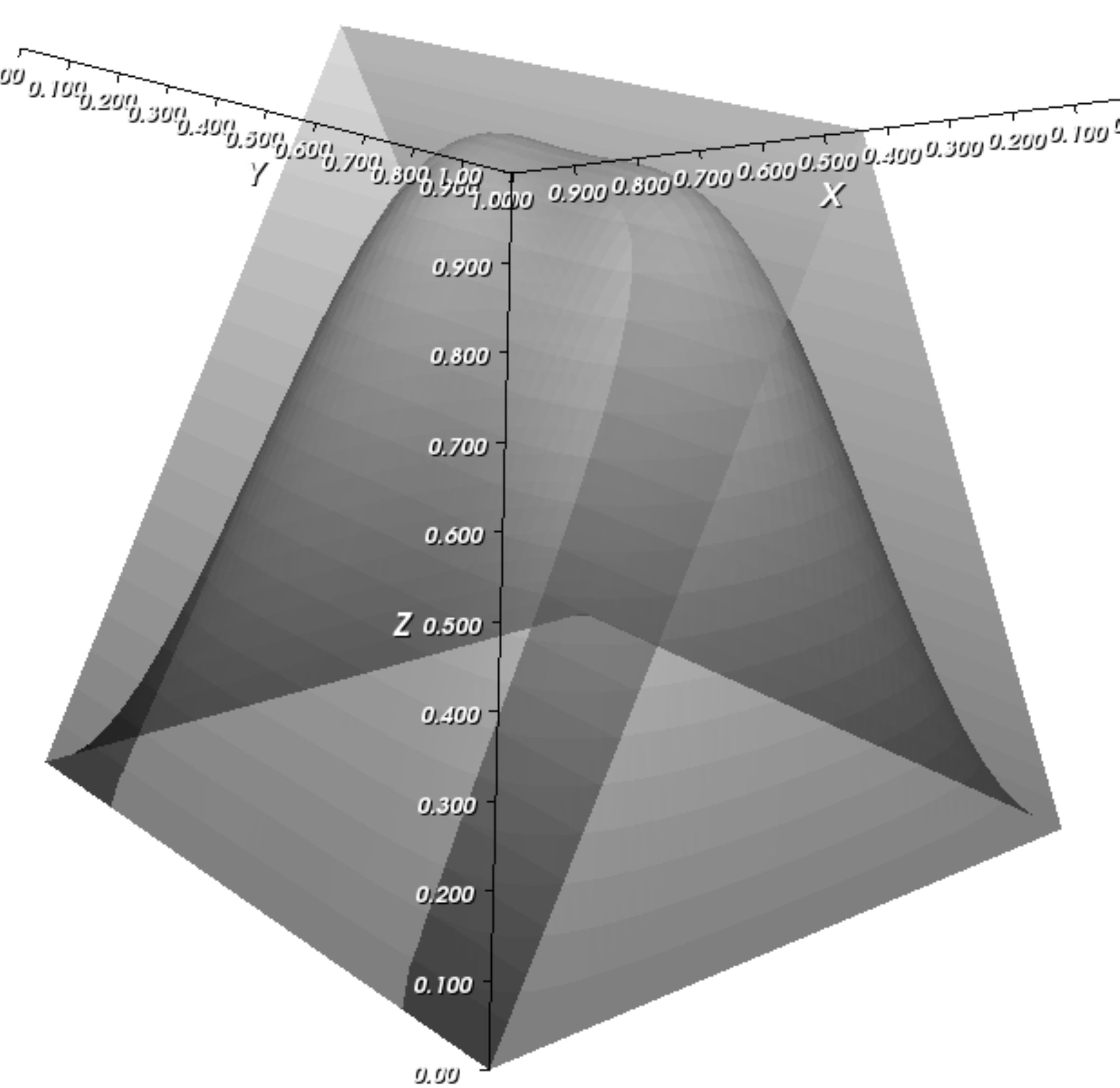}\quad
\includegraphics[scale=0.29]{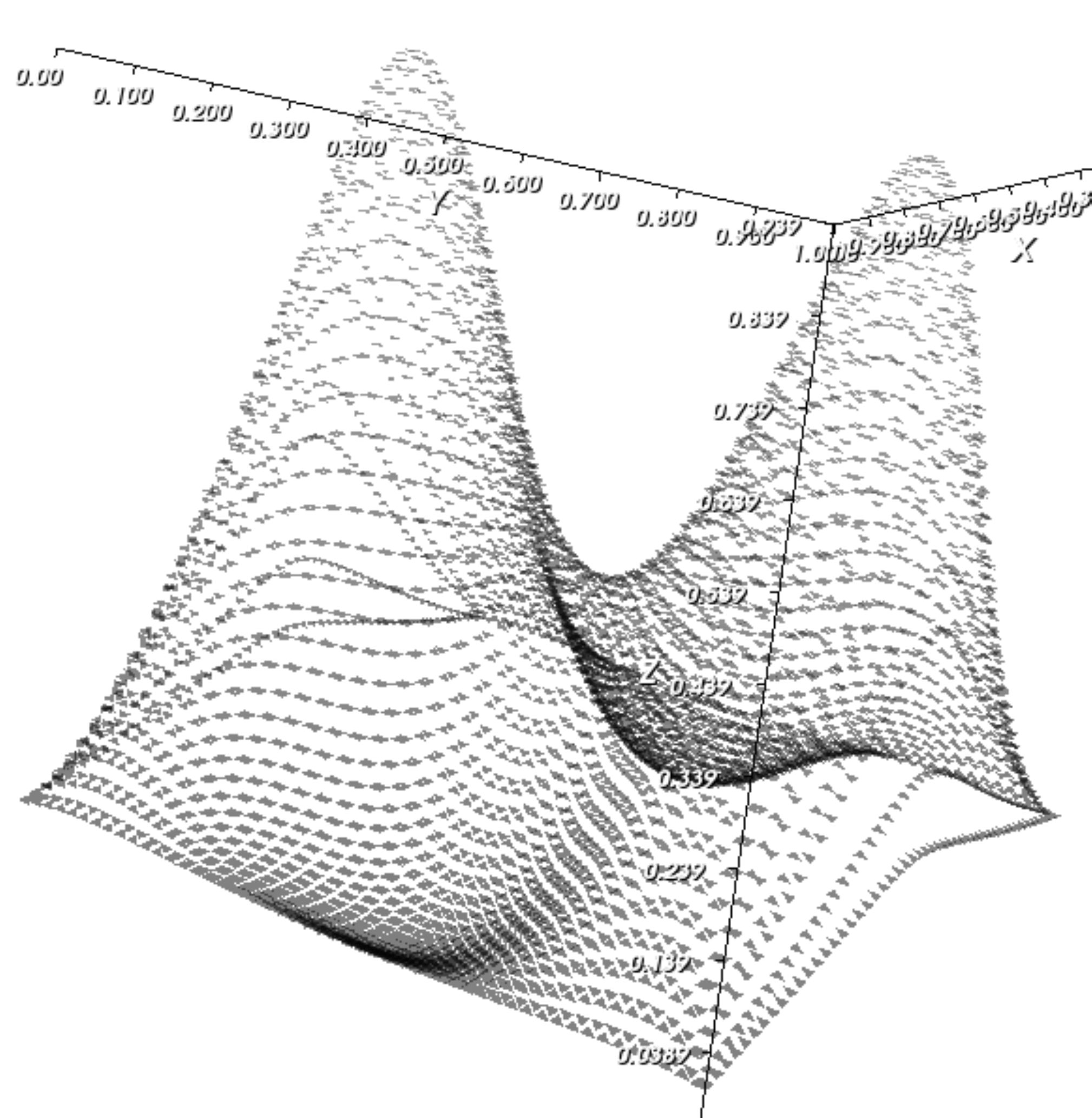}
\caption{A comparison between the target and the optimal profile (left) for $\mu=0.0001$. The control $v$ on the right.}
\label{tres}       
\end{figure}

To check the performance of the algorithm, we will start by treating a classic, well-known, linear situation in which we specifically take $\Omega=(0, 1)\times(0, 1)$, $\lambda=1$, $\phi(u)=-1$, $\overline u(x, y)=\min(x, 1-x)$, and play with decreasing values for $\mu$. 
Note that the existence hypotheses of Theorem \ref{existencia} are clearly fulfilled for $l(u)=-1$. 

To solve the various linear PDE problems involved, a software like FreeFem++ (\cite{freefem}) looks like an ideal choice for the non-expert. We show the numerical results for the three values $\mu=1.e-02$ (Figure \ref{uno}), $\mu=1.e-03$ (Figure \ref{dos}), and $\mu=1.e-04$ (Figure \ref{tres}, error=0.0061237 ($L^2$-norm of the gradient of the residual $w$)). Figure \ref{tres-dos} contains a level-curve map for the residual function for $\mu=0.0001$.

\begin{figure}[b]
\includegraphics[scale=0.35]{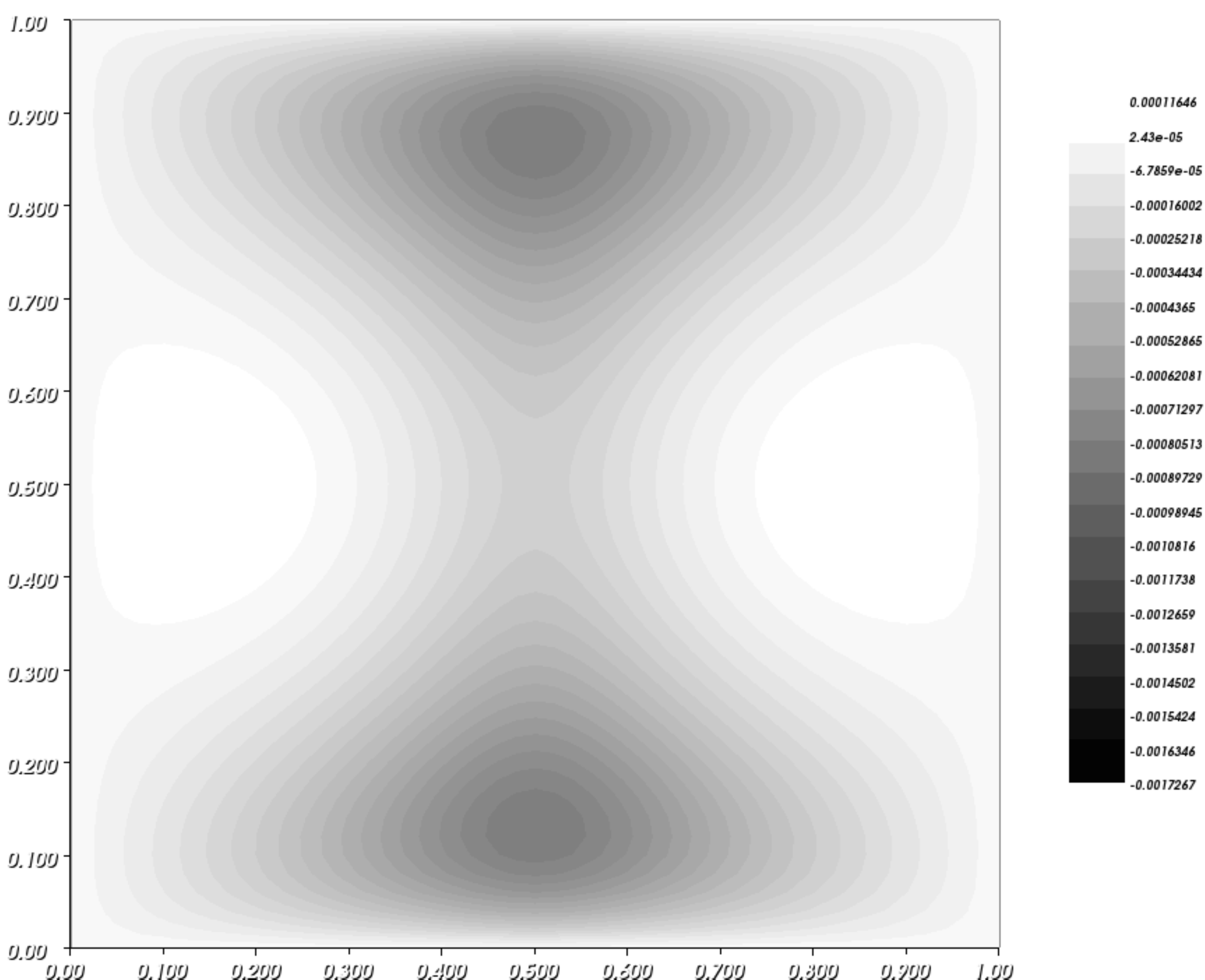}
\caption{The contour lines for the residual for $\mu=0.0001$. The size ($L^2$-norm of the gradient) is 0.0061237.}
\label{tres-dos}       
\end{figure}

We next treat the same situations with the same values of the parameters but for the non-linearity $\phi(u)=(u-2)^3$. 
Again the existence hypotheses of Theorem \ref{existencia} are correct since the four-degree polynomial $-u(u-2)^3$ is clearly bounded from above by some fixed constant.

The results of the simulations can be seen in Figure \ref{cuatro} for the value $\mu=1.e-02$, and error=0.010073; in Figure \ref{cinco} for $\mu=1.e-03$ with error=0.0083419; and in Figure \ref{seis} for $\mu=1.e-04$ and error=0.00541891.  Again Figure \ref{seis-dos} depicts a level-curve map for the residual function $w$. 

\begin{figure}[b]
\includegraphics[scale=0.3]{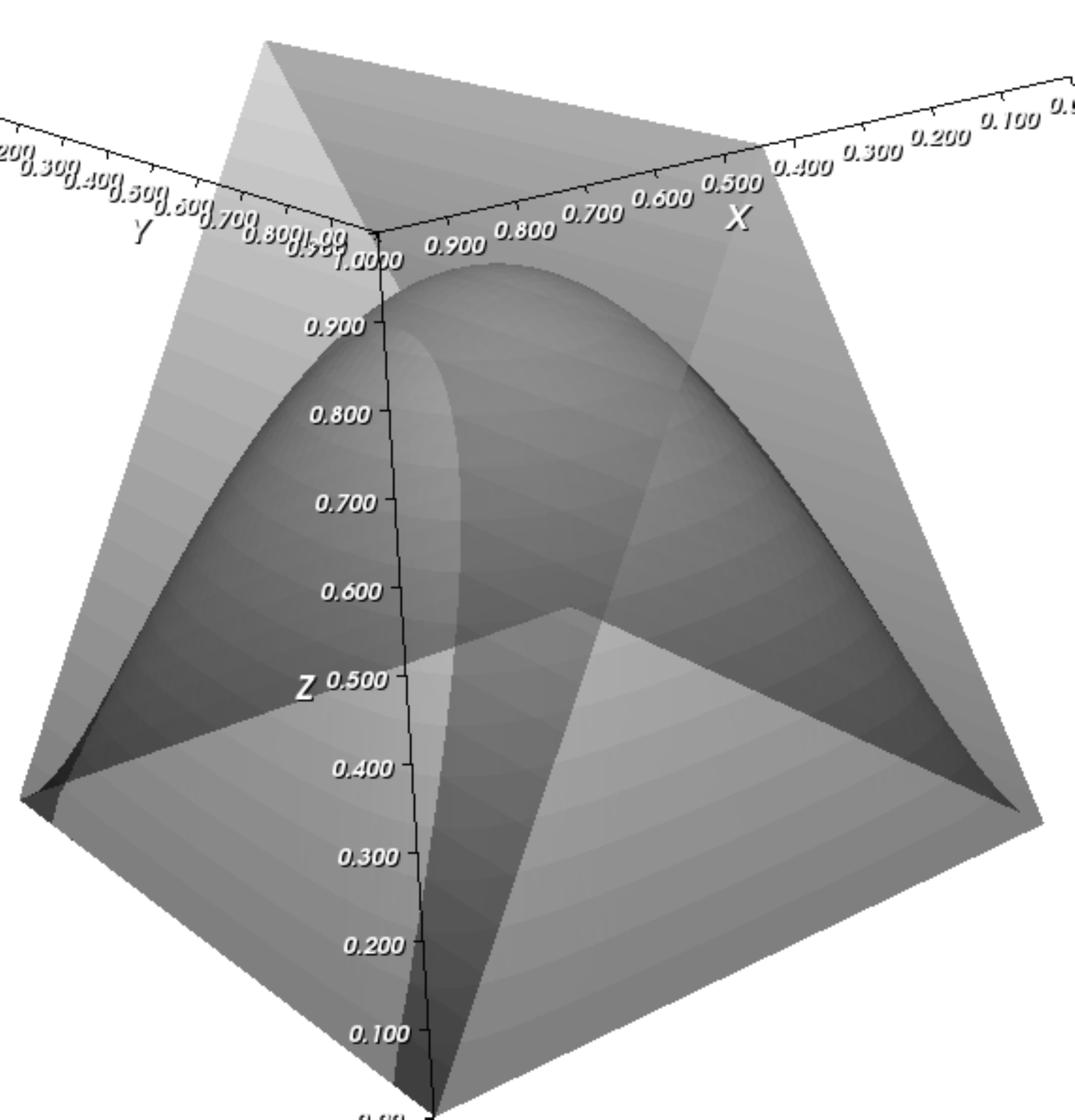}\quad
\includegraphics[scale=0.3]{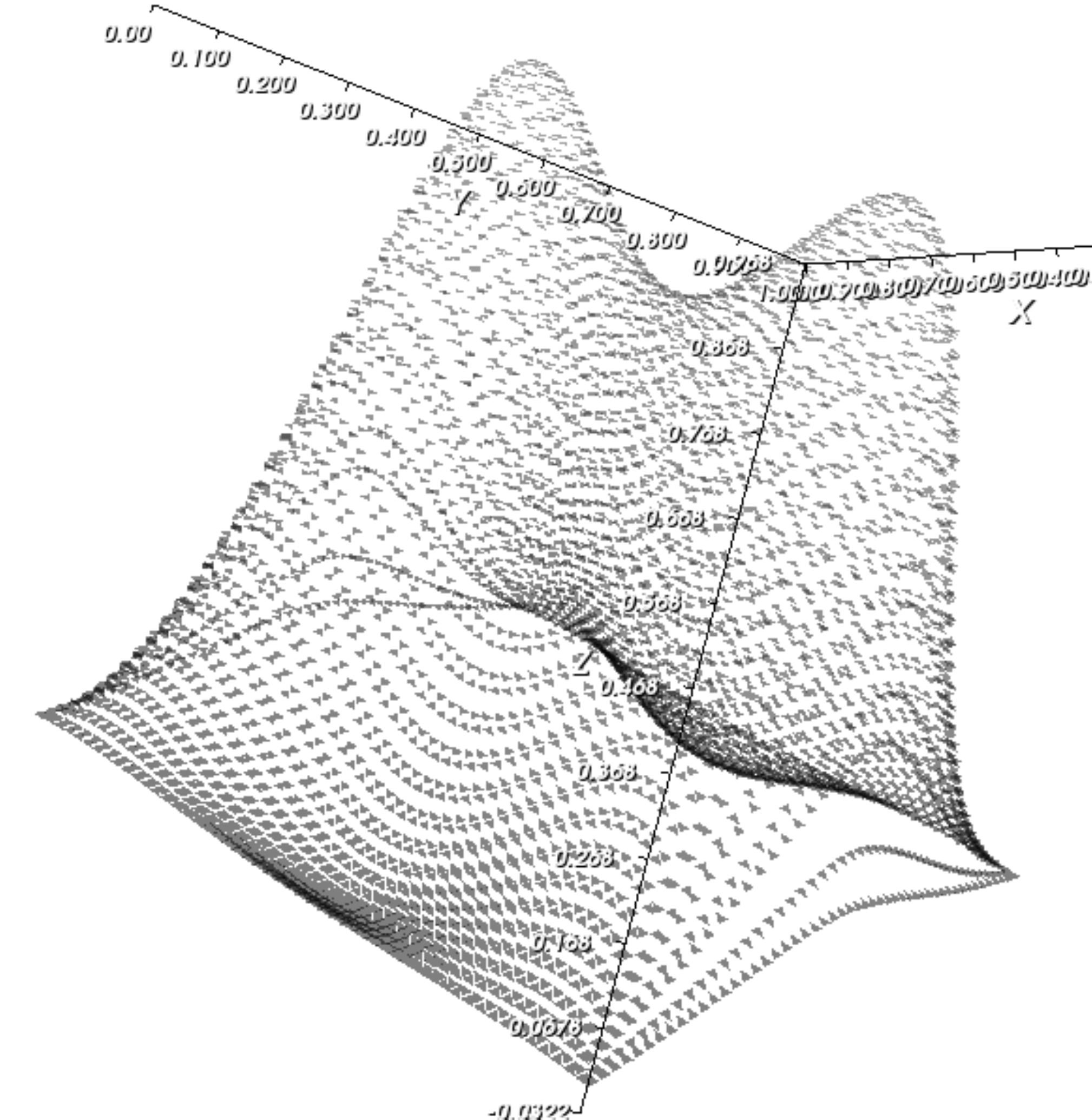}
\caption{A comparison between the target and the optimal profile (left) for the non-linear problem and $\mu=0.01$. The control $v$ on the right.}
\label{cuatro}       
\end{figure}

\begin{figure}[b]
\includegraphics[scale=0.29]{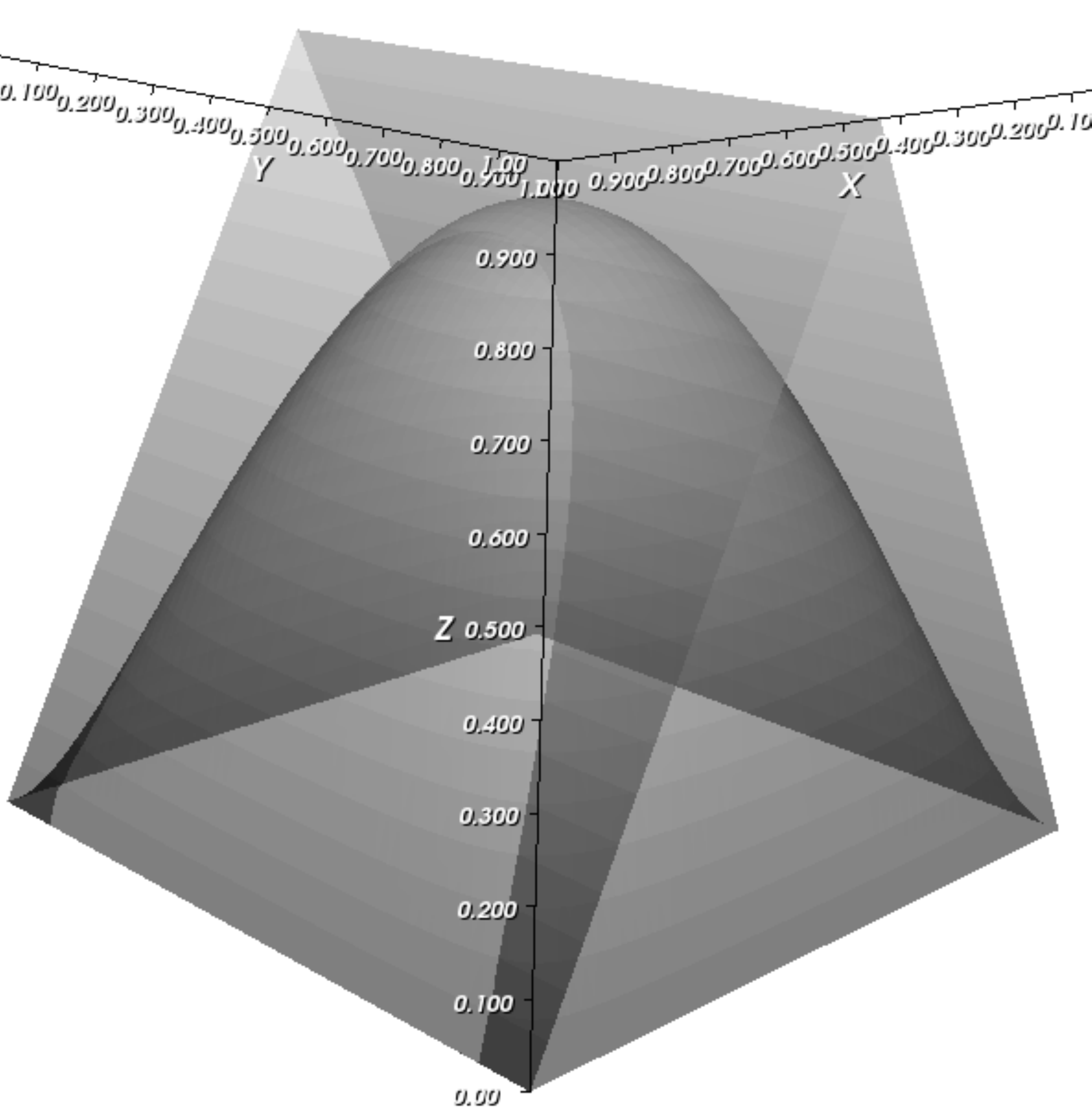}\quad
\includegraphics[scale=0.29]{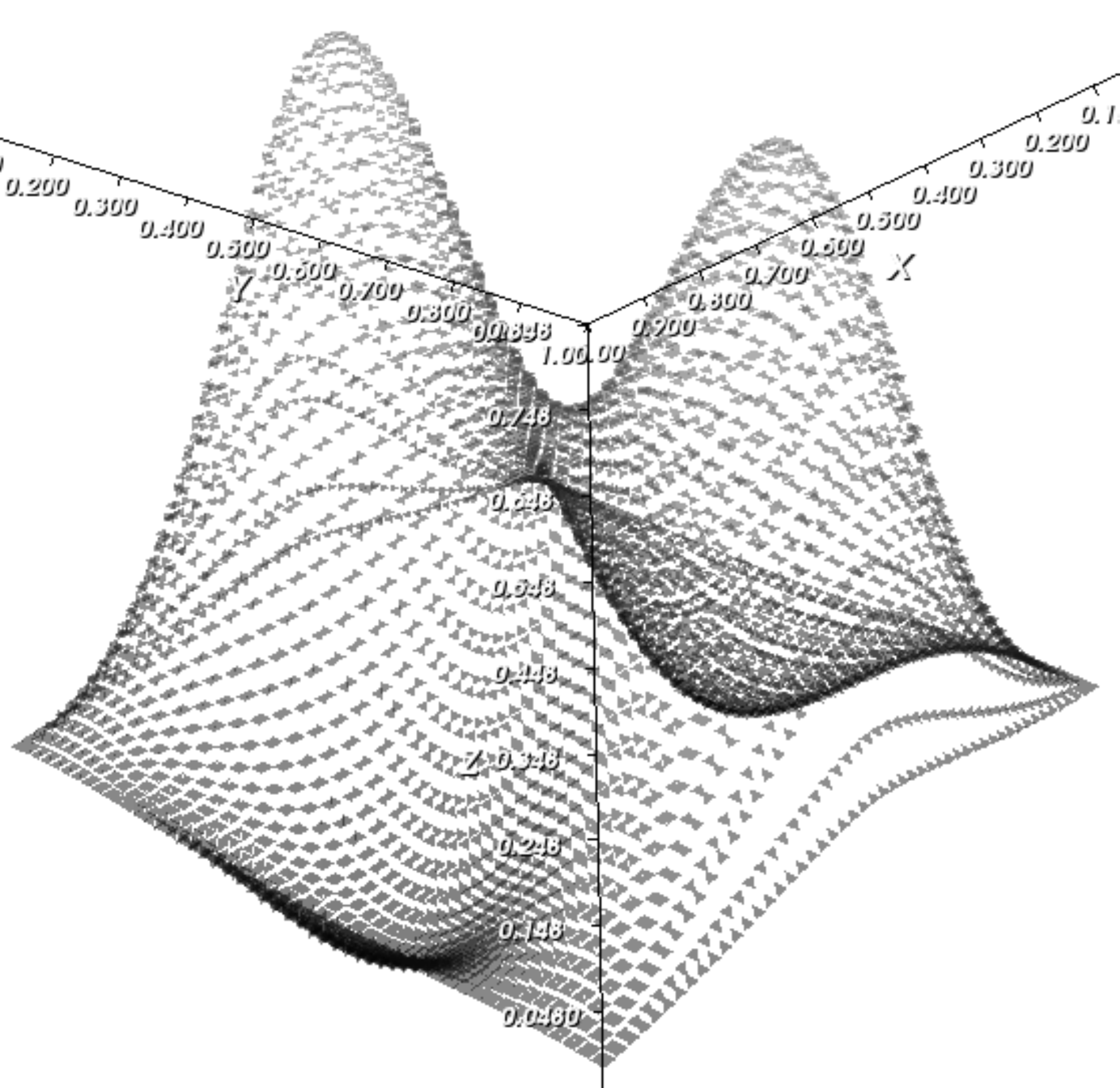}
\caption{A comparison between the target and the optimal profile (left) for the non-linear problem and $\mu=0.001$. The control $v$ on the right.}
\label{cinco}       
\end{figure}

\begin{figure}[b]
\includegraphics[scale=0.29]{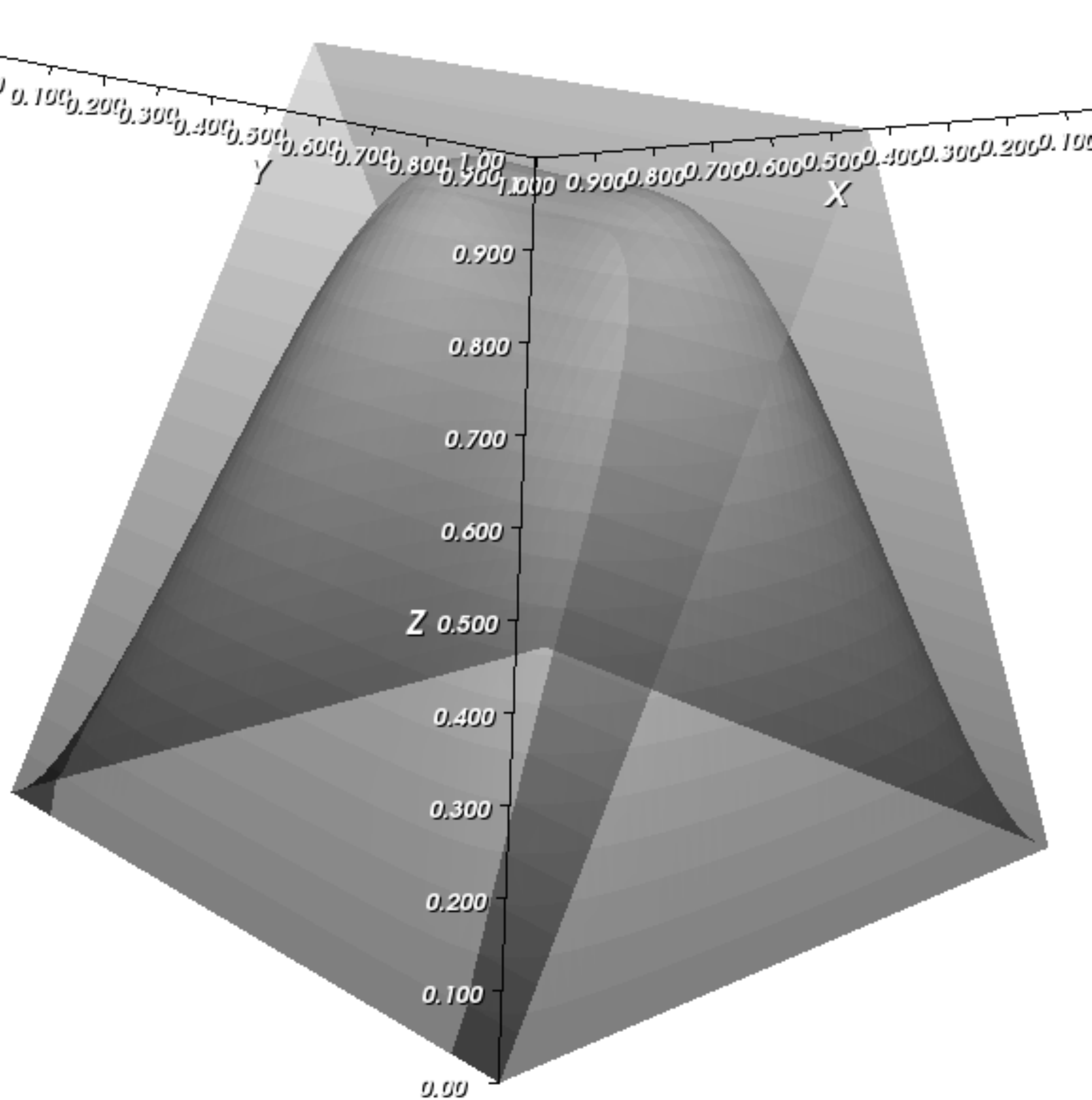}\quad
\includegraphics[scale=0.29]{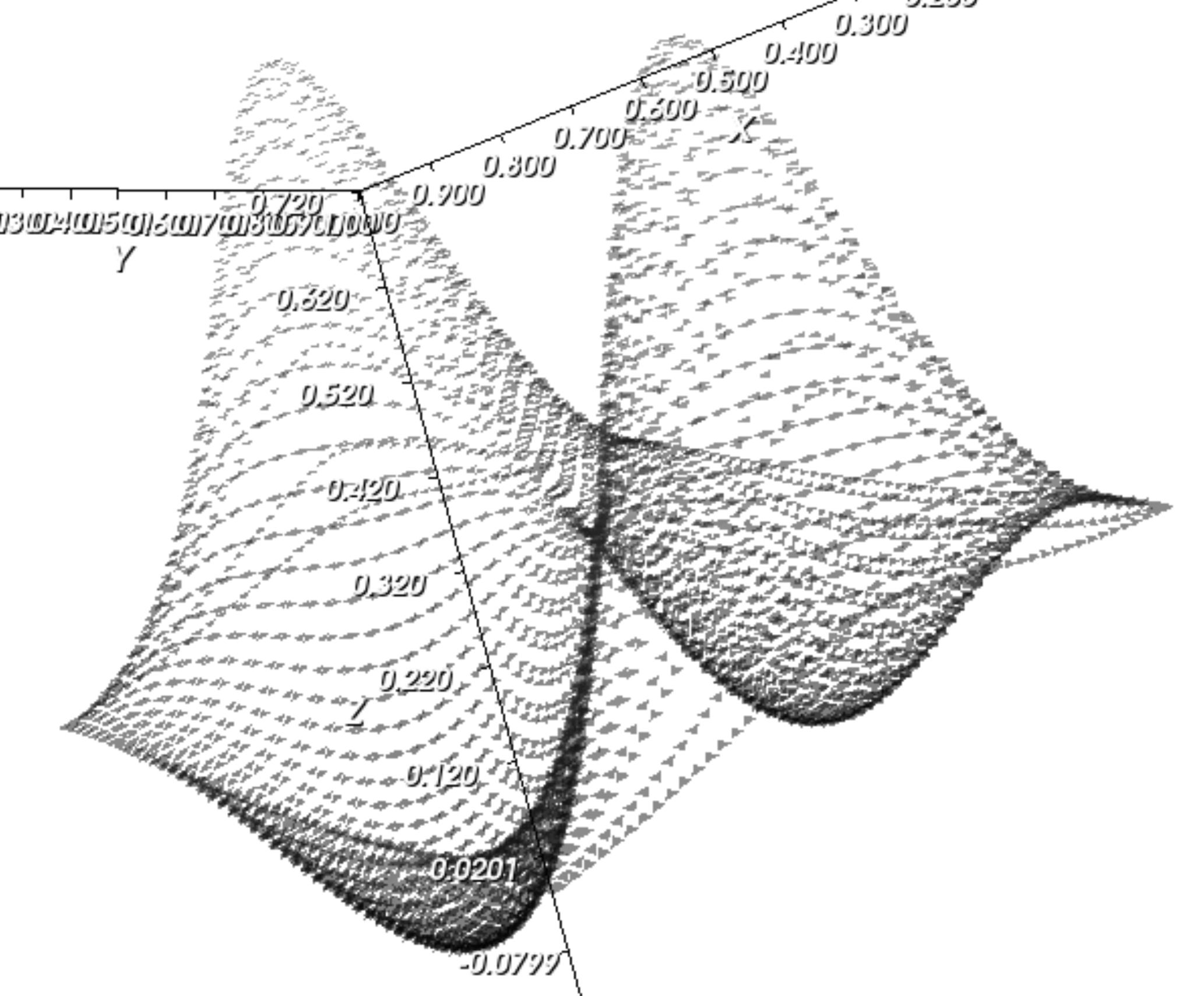}
\caption{A comparison between the target and the optimal profile (left) for the non-linear problem and $\mu=0.0001$. The control $v$ on the right.}
\label{seis}       
\end{figure}

\begin{figure}[b]
\includegraphics[scale=0.35]{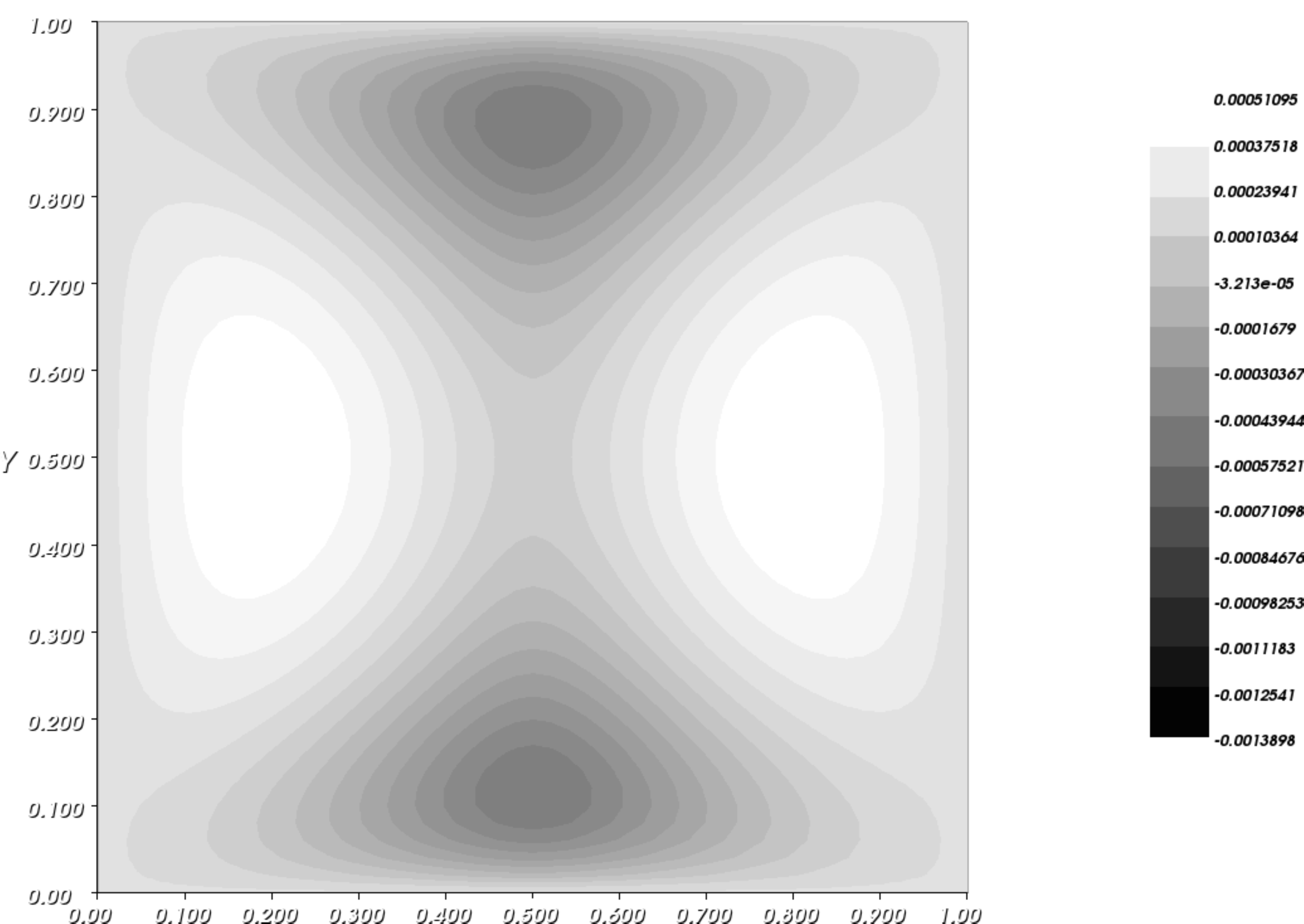}
\caption{The contour lines for the residual for $\mu=0.0001$.}
\label{seis-dos}       
\end{figure}

\section{State constraints}\label{cuatro}
We now change our model problem to
\begin{equation}\label{restest}
\hbox{Minimize in }(u, v)\in\A:\quad \int_\Omega\left(\frac12|u(\bx)-\overline u(\bx)|^2+\frac{\lambda}2|\nabla w(\bx)|^2\right)\,d\bx
\end{equation}
subject to
$$
-\dv[\nabla u(\bx)+\nabla w(\bx)]=v(\bx)\hbox{ in }\Omega,\quad
w=0\hbox{ on }\partial\Omega,
$$
and
$$
\A=\{(u, v)\in H^1_0(\Omega)\times L^\infty(\Omega): u(\bx)\le0, v_-(\bx)\le v(\bx)\le v_+(\bx)\}.
$$
The functions $v_-$ and $v_+$ are fixed $L^\infty(\Omega)$-functions such that $v_-(\bx)\le v_+(\bx)$ for a.e. $\bx\in\Omega$. Note that, by putting $v_-(\bx)=v_+(\bx)=0$ for points $\bx$ in a certain subset of $\Omega$, we can localize the effect of the control. 

Since in this situation, we are to enforce convex, pointwise constraints for both sets of variables, state and control, but in a linear fashion, the existence theorem Theorem \ref{existencia} is not compromised in the least, and we still have a unique optimal pair $(u, v)\in\A$.  
\begin{proposition}
Under the hypotheses indicated, there is a unique optimal pair $(u, v)\in\A$ for problem \eqref{restest}.
\end{proposition}
\begin{proof}
For the proof, simply notice that we are dealing with a strictly-convex, coercive, quadratic functional over a convex set of competing functions. 
By using $u$ itself as a test function in the state equation and following the same estimates as in the case without constraints, one finds that
$$
\|\nabla u\|_{L^2(\Omega)}\le \|\nabla w\|_{L^2(\Omega)}+M\le 
\sqrt{\frac{2 I[u, v]}\lambda}+M
$$
where this time $M$ is a constant depending on the $L^\infty$-norm of $v_-$ and $v_+$, as well as in Poincar\'e's inequality. This estimate, together with the uniform bound for feasible functions $v$, imply that $I$ is coercive in $\A$. 

The strict convexity comes directly from the linearity of the state equation, and the strict convexity of $I$ itself. 

The existence of a unique solution is then a classical result. 
\end{proof}
Optimality can then be expressed in the form of variational inequalities by considering perturbations
$$
u\mapsto u+\epsilon(U-u),\quad v\mapsto v+\epsilon(V-v)
$$
for arbitrary pairs $(U, V)\in\A$, and demand one-sided conditions. Computations are similar to the previous situation by simply replacing $U$ by $U-u$, and $V$ by $V-v$,
namely
\begin{gather}
\left.\frac{d}{d\epsilon}I[(u, v)+\epsilon(U-u, V-v)]\right|_{\epsilon=0}=\nonumber\\
\int_\Omega[(u-\overline u)(U-u)+\lambda((V-v)w-(\nabla U-\nabla u)\cdot\nabla w)]\,d\bx\ge0,\nonumber
\end{gather}
for every such pair $(U, V)\in\A$. Note that this time the equation
$$
-\dv(\nabla U-\nabla u+\nabla W)=V-v
$$
furnishes the variation $W$ on $w$ produced by the perturbations on $u$ and $v$ given above. 
We can reinterpret this family of inequalities as follows. 
If we do not perturb $u$ and take $U=u$, then it is elementary to realize that we can put
$$
V=\frac{v_++v_-}2+\sgn{w}\frac{v_+-v_-}2.
$$
Conversely, if we do not pertub $v$ and take $V=v$, then the following statement yields, in a specific way, how to select $U$. Recall that $w$ is the defect associated with the feasible pair $(u, v)$ according to the state law right after \eqref{restest}. 

\begin{lemma}
For given $u$, the unique solution $U$ of the obstacle problem
$$
\hbox{Minimize in }U\in H^1_0(\Omega):\quad \int_\Omega\left(\frac12|\nabla u-\nabla U|^2+(u-\overline u)U-\lambda\nabla w\cdot\nabla U\right)\,d\bx
$$
under $U\le0$, is a descent direction of $I$ at $u$. 
\end{lemma}
\begin{proof}
To begin with, it is true that the variational problem for $U$ admits a unique solution. This is a classical result on variational inequalities and obstacle problems (\cite{kinder}).

The proof of the statement amounts to writing the one-sided condition for a variation of the specific form
$$
U\mapsto U+\epsilon(u-U).
$$
Indeed, if $U$ is the minimizer of the problem in the statement (whose existence is a direct consequence of very classical results \cite{kinder}), replacing $U$ by $U+\epsilon(u-U)$ in the functional in the statement, differentiating with respect to $\epsilon$ and setting $\epsilon=0$, we find that
$$
\int_\Omega[-(\nabla u-\nabla U)\cdot(\nabla u-\nabla U)+(u-\overline u)(u-U)-
\lambda\nabla w\cdot(\nabla u-\nabla U)]\,d\bx\ge0.
$$
The derivative of $I$ at $(u, v)$ in the direction $(U-u, \bcero)$ computed earlier becomes
$$
\int_\Omega[(u-\overline u)(U-u)-\lambda(\nabla U-\nabla u)\cdot\nabla w)]\,d\bx\le-\int_\Omega|\nabla u-\nabla U|^2\,d\bx\le0.
$$
\end{proof}
\subsection{The numerical implementation}
Though the numerical implementation can make use of the previous lemma for a typical descent algorithm, we would like to explore the possibility of implementing a direct mechanism taking care of pointwise constraints through appropriate exponential barriers as in \cite{pedregal}. 

The procedure is in need of several auxiliary variables $a(\bx)$, $b^-(\bx)$, $b^+(\bx)$ that, somehow, play the role of multipliers, and goes as follows:
\begin{enumerate}
\item Initialization. Take $a_0(\bx)$, $b^-_0(\bx)$, $b^+_0(\bx)$ arbitrary but strictly positive all over the domain $\Omega$ (they, in particular, can be taken to be positive constants). 
\item Main iterative step until convergence. Suppose we have $a_{j-1}(\bx)$, $b^-_{j-1}(\bx)$, $b^+_{j-1}(\bx)$, then:
\begin{enumerate}
\item Approximate the unique solution of the unconstrained problem that consists of minimizing the funcional 
\begin{gather}
\int_\Omega\left(\frac12|u(\bx)-\overline u(\bx)|^2+\frac{\lambda}2|\nabla w(\bx)|^2+\exp(a_{j-1}(\bx)u(\bx))\right.\nonumber\\
+\left.\exp(b^-_{j-1}(\bx)(v_-(\bx)-v(\bx)))+\exp(b^+_{j-1}(\bx)(v(\bx)-v_+(\bx)))\right)\,d\bx\nonumber
\end{gather}
freely in pairs $(u, v)\in H^1_0(\Omega)\times L^2(\Omega)$
subject to
$$
-\dv[\nabla u(\bx)+\nabla w(\bx)]=v(\bx)\hbox{ in }\Omega,\quad
w=0\hbox{ on }\partial\Omega.
$$
This is a problem of the complementary kind indicated at the beginning of Section \ref{dos}, \eqref{statee}. The convexity conditions on the variables $(u, v)$ involved in the exponential terms ensure that there is a unique optimal pair $(u_j(\bx), v_j(\bx))$. 
\item Check if the three products occurring in the exponentials, namely
$$
a_{j-1}(\bx)u_j(\bx),\quad b^-_{j-1}(\bx)(v_-(\bx)-v_j(\bx)),\quad
b^+_{j-1}(\bx)(v_j(\bx)-v_+(\bx)),
$$
vanish, or are sufficiently small. If they are so, take the pair 
$$
(u_j(\bx), v_j(\bx))
$$ 
as a good approximation of the underlying optimal solution of the constrained problem. If at least one of the products is not, update coefficients according to the rule
\begin{gather}
a_j(\bx)=a_{j-1}(\bx)\exp(a_{j-1}(\bx)u_j(\bx)),\nonumber\\
b^-_j(\bx)=b^-_{j-1}(\bx)\exp(b^-_{j-1}(\bx)(v_-(\bx)-v_j(\bx))),\nonumber\\
b^+_j(\bx)=b^+_{j-1}(\bx)\exp(b^+_{j-1}(\bx)(v_j(\bx)-v_+(\bx))),\nonumber
\end{gather}
and proceed to the previous step. 
\end{enumerate}
\end{enumerate}
It is easy to understand, intuitively, the effect of the update rule for the triplet of coefficients. Beyond that understanding, we have the following convergence result.

\begin{theorem}\label{convergencia}
Let $\{(u_j, v_j)\}\in H^1_0(\Omega)\times L^2(\Omega)$ be the sequence of iterates produced by the previous scheme.  If there is a pair $(u, v)\in H^1_0(\Omega)\times L^2(\Omega)$ such that
$$
(u_j, v_j)\to(u, v)\hbox{ pointwise},
$$
and the three products
$$
a_{j-1}(\bx)u_j(\bx),\quad  b^-_{j-1}(\bx)(v_-(\bx)-v_j(\bx)),\quad 
b^+_{j-1}(\bx)(v_j(\bx)-v_+(\bx)),
$$
converge to zero for a.e. $\bx\in\Omega$, then $(u, v)$ is the solution of the underlying constrained problem \eqref{restest} (in particular it complies with the corresponding pointwise constraints).
\end{theorem}
\begin{proof}
For given, non-negative, measurable $a(\bx)$, $b^-(\bx)$, $b^+(\bx)$, the associated, unconstrained optimal control problem over the space $H^1_0(\Omega)\times L^2(\Omega)$ with functional
\begin{gather}
I_{(a, b^-, b^+)}[u, v]=\int_\Omega\left(\frac12|u(\bx)-\overline u(\bx)|^2+\frac{\lambda}2|\nabla w(\bx)|^2+\exp(a(\bx)u(\bx))\right.\nonumber\\
+\left.\exp(b^-(\bx)(v_-(\bx)-v(\bx)))+\exp(b^+(\bx)(v(\bx)-v_+(\bx)))\right)\,d\bx\nonumber
\end{gather}
under the state equation
$$
-\dv[\nabla u(\bx)+\nabla w(\bx)]=v(\bx)\hbox{ in }\Omega,\quad
w=0\hbox{ on }\partial\Omega.
$$
admits a unique optimal pair $(u, v)$ which, obviously, will depend on the triplet $(a, b^-, b^+)$. As pointed out above, this is a consequence of our results in Section \ref{dos}. Because for arbitrary triplets $(a, b^-, b^+)$, the exponential terms are non-negative, the family of functionals 
$$
\{I_{(a, b^-, b^+)}: (a, b^-, b_+)\ge0\}
$$
are uniformly bounded from below by the one without those exponential terms which is the initial functional $I$. Note that the state law is independent of the triplet. Recall that the problem whose optimal solution we would like to approximate, namely
\begin{equation}
\hbox{Minimize in }(u, v)\in H^1_0(\Omega)\times L^2(\Omega):\quad \int_\Omega\left(\frac12|u(\bx)-\overline u(\bx)|^2+\frac{\lambda}2|\nabla w(\bx)|^2\right)\,d\bx
\end{equation}
subject to
$$
-\dv[\nabla u(\bx)+\nabla w(\bx)]=v(\bx)\hbox{ in }\Omega,\quad
w=0\hbox{ on }\partial\Omega,
$$
is coercive, and admits a unique optimal solution according to our discussion in Section \ref{dos}. As a consequence, the sequence of optimal pairs $(u_j, v_j)$ corresponding to a given sequence of functionals $I_{(a_j, b^-_j, b^+_j)}$ will admit a weakly converging subsequence to some $(u, v)$ in $H^1_0(\Omega)\times L^2(\Omega)$. On the other hand, note how for feasible triplets $(a, b^-, b^+)$ and feasible pairs $(u, v)$, the arguments in the exponential are non-positive, and so those exponential belong to any $L^p$ space (provided that $\Omega$ is bounded). 

In particular, if $(u_j, v_j)$ is the full sequence of iterates generated successively through our proposed algorithm, with $\{(a_j, b^-_j, b^+_j)\}$ the corresponding sequence of iterates for which the hypotheses of the statement hold, the limit pair $(u, v)\in H^1_0(\Omega)\times L^2(\Omega)$. 
\begin{enumerate}
\item For a.e. $\bx\in\Omega$, we claim that
$$
u(\bx)\le0,\quad v_-(\bx)\le v(\bx)\le v_+(\bx).
$$
Let $\bx$ be a certain point, where, in addition to the the pointwise convergences in the statement of the theorem, we have, say, $u(\bx)>0$ (the other two possibilities are argued exactly in the same way). The three conditions
$$
u(\bx)>0, \quad u_j(\bx)\to u(\bx),\quad a_{j-1}(\bx)u_j(\bx)\to0,
$$
together with the update rule
$$
a_j(\bx)=a_{j-1}(\bx)\exp(a_{j-1}(\bx)u_j(\bx))
$$
are clearly inconsistent (at least for large $j$). Recall that $a_{j-1}(\bx)$ is always strictly positive (though it may be very small). Note how this update rule implies that the sequence $\{a_j(\bx)\}$ is an increasing sequence of positive numbers because the arguments in the exponentials are strictly positive. The same is true for the pointwise constraints for $v$. 
\item The limit pair $(u, v)$ is the unique optimal solution of our initial constrained optimal control problem. Suppose not, and let $(\overline u, \overline v)$ be a feasible pair for our constrained problem such that 
\begin{equation}\label{desigu}
I[\overline u, \overline v]< I[u, v]
\end{equation}
where $I[u, v]$ is the cost functional for our optimal control problem. Recall that $I_{(a, b^-, b^+)}$ is just $I$ plus the three exponential terms corresponding to the triplet $(a, b^-, b^+)$. Let $j$ be large. If we further put
$$
I^{exp}_{(a, b^-, b^+)}[u, v]=\int_\Omega[(\exp(au)+\exp(b^-(v_--v))+
\exp(b^+(v-v^+))]\,d\bx,
$$
then
$$
I_{(a_j, b^-_j, b^+_j)}[\overline u, \overline v]=I[\overline u, \overline v]+
I^{exp}(a_j, b^-_j, b^+_j)[\overline u, \overline v].
$$
Through this identity we realize that (bear in mind \eqref{desigu})
\begin{equation}\label{desigudos}
I_{(a_j, b^-_j, b^+_j)}[\overline u, \overline v]<I[u, v]+3\le \liminf_{k\to\infty} I_{(a_k, b^-_k, b^+_k)}[u_k, v_k],
\end{equation}
because, on the one hand,
$$
\exp(a_j\overline u)\le 1=\lim_{k\to\infty}\exp(a_ku_k),
$$
and similarly for the inequalities involving $b^\pm_j$ and $\overline v$, and, on the other, we have the weak lower semicontinuity
$$
I(u, v)\le\liminf_{k\to\infty} I(u_k, v_k).
$$
Recall that $(u_k, v_k)\to(u, v)$ weakly in $H^1_0(\Omega)\times L^2(\Omega)$. But the resulting inequality in \eqref{desigudos} yields that there is a sufficiently large $j$ for which
$$
 I_{(a_j, b^-_j, b^+_j)}[\overline u, \overline v]<I_{(a_j, b^-_j, b^+_j)}[u_j, v_j],
 $$
 but this is not possible because, by definition, the pair $(u_j, v_j)$ is the unique optimal solution of the optimal control problem over $H^1_0(\Omega)\times L^2(\Omega)$ with cost functional $I_{(a_j, b^-_j, b^+_j)}$ and the same linear underlying state law. This contradiction shows that indeed $(u, v)$ is our optimal pair. 
\end{enumerate}
\end{proof}

\begin{figure}[b]
\includegraphics[scale=0.29]{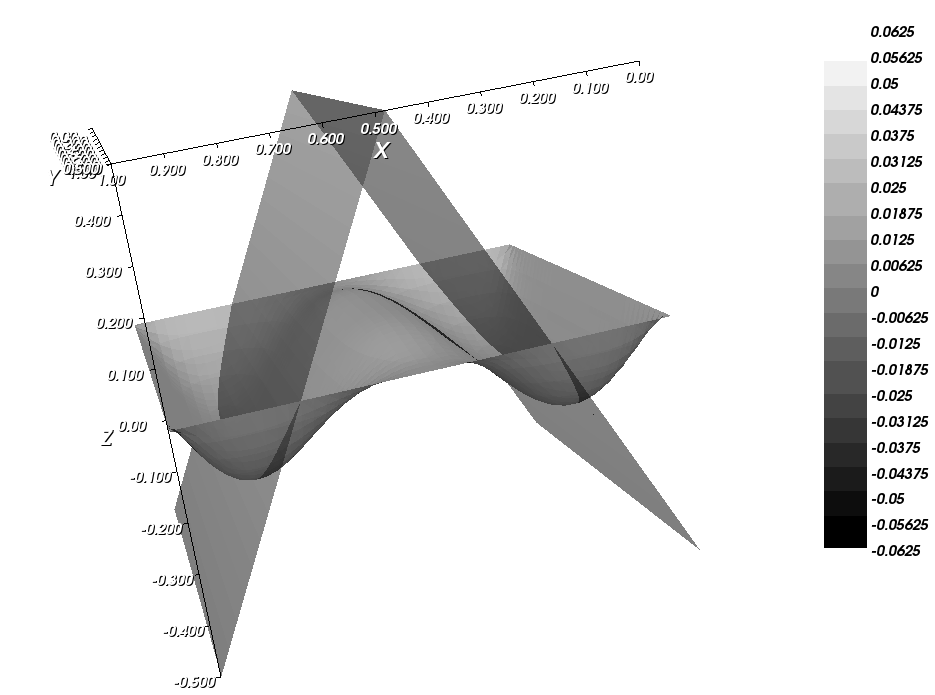}
\caption{A comparison between the target and the optimal profile $u$ for the state constrained example.}
\label{final}       
\end{figure}

\begin{figure}[b]
\includegraphics[scale=0.29]{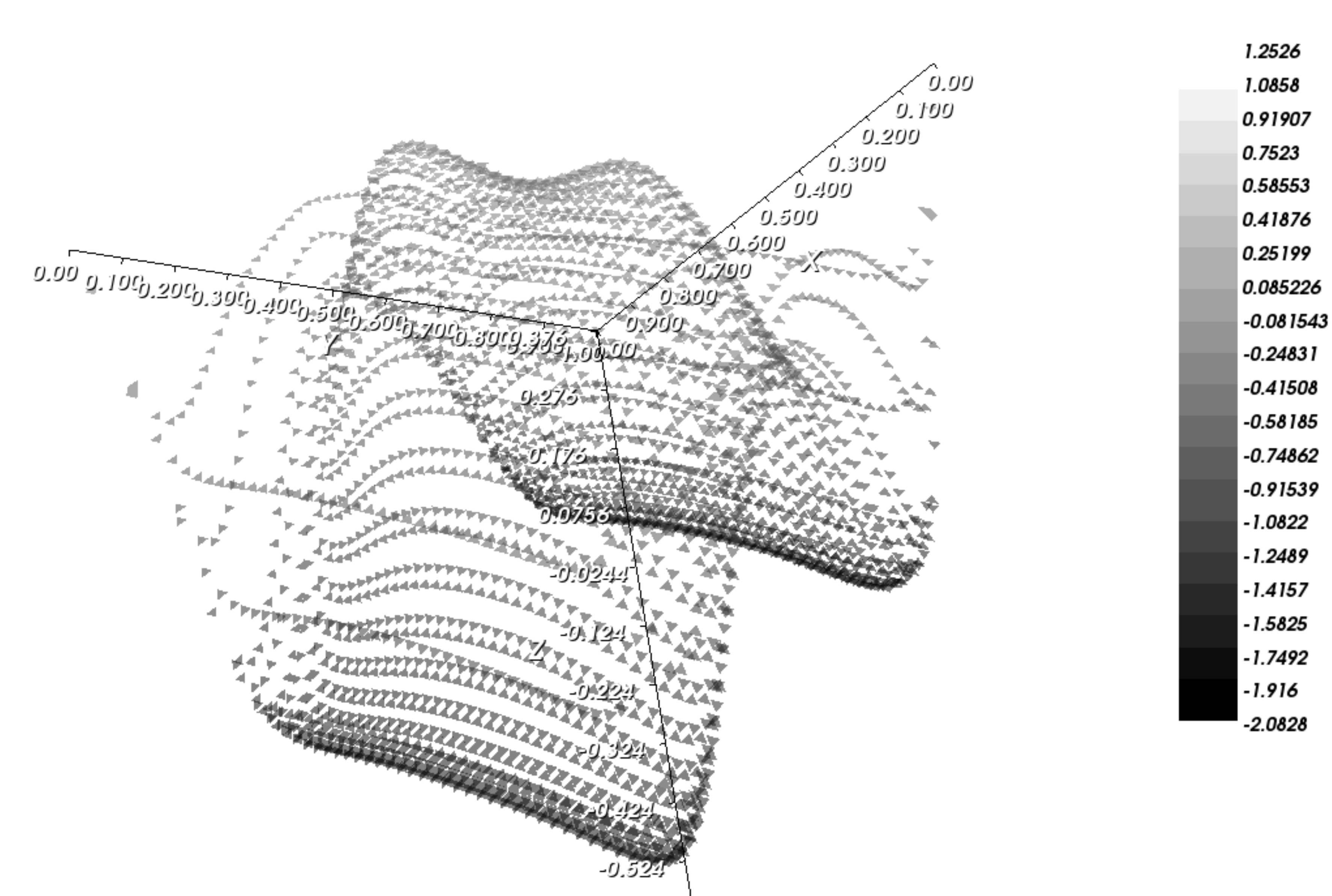}
\caption{Control function $v$ for the state constrained case.}
\label{finall}       
\end{figure}

\begin{figure}[b]
\includegraphics[scale=0.29]{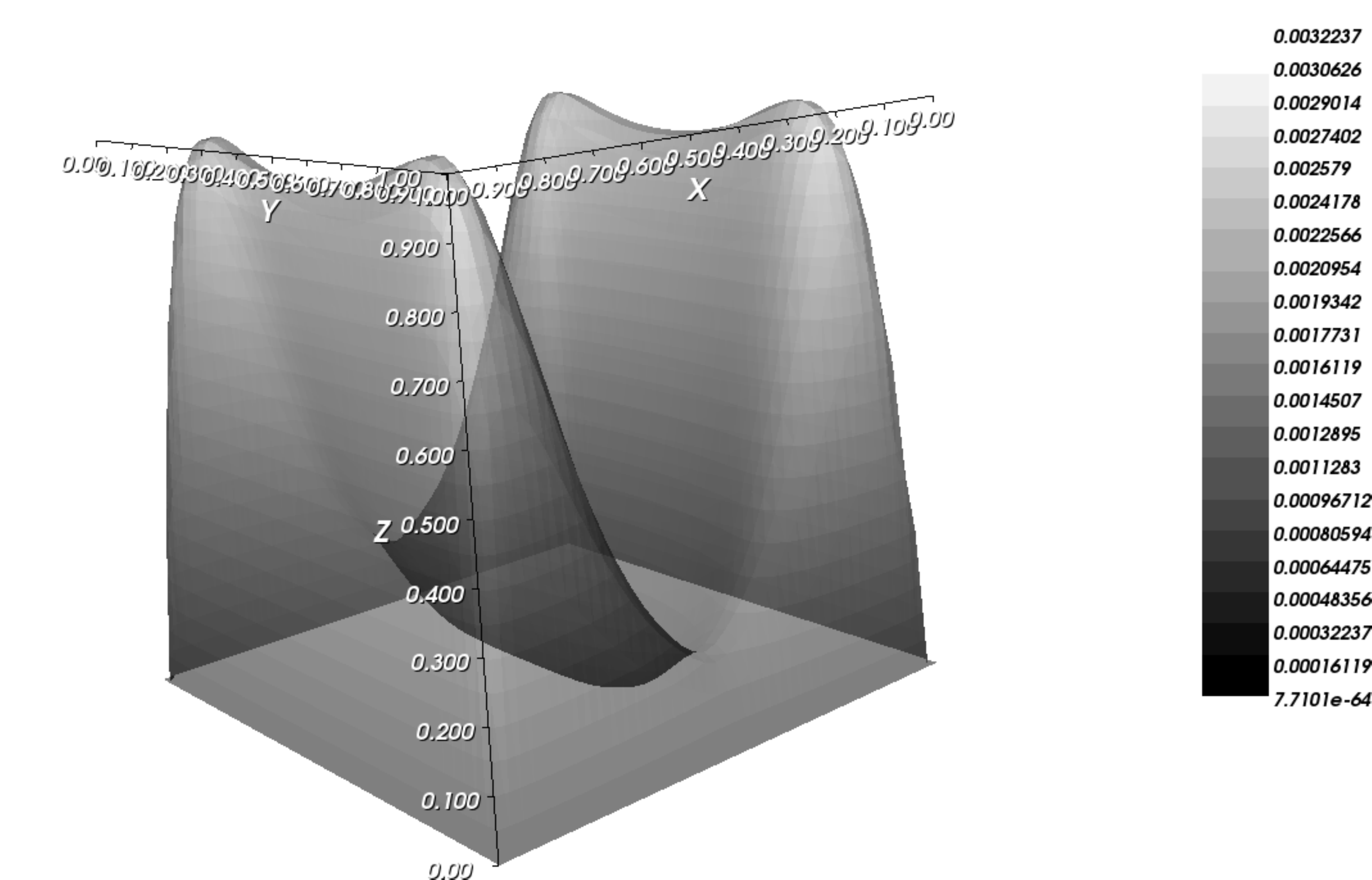}
\caption{Residual function $w$ for the constrained example.}
\label{finalll}       
\end{figure}

Just for the sake of illustration, we present one example of the situation just treated. As usual $\Omega$ is the unit square, and we take
$$
\overline u(x, y)=\frac14(\min(x, 1-x)-\frac14),\quad 
v_+\equiv 5., v_-\equiv-3.,\quad \lambda=.1.
$$
As starting values for the auxiliary functions $a$, $b^-$, $b^+$ we set
$$
a_0\equiv b^-_0\equiv b^+_0\equiv.1.
$$
The computed state $u$ is shown in Figure \ref{final}, while the approximated control $v$ can be checked on Figure \ref{finall}. As stated in Theorem \ref{convergencia}, the vanishing pointwise limit of  the three products 
$$
a_{j-1}(\bx)u_j(\bx),\quad  b^-_{j-1}(\bx)(v_-(\bx)-v_j(\bx)),\quad 
b^+_{j-1}(\bx)(v_j(\bx)-v_+(\bx)),
$$
to zero is a certificate of convergence. See these three products for this example in Figures \ref{certa}, \ref{certb-}, \ref{certb+}. 

\begin{figure}[b]
\includegraphics[scale=0.29]{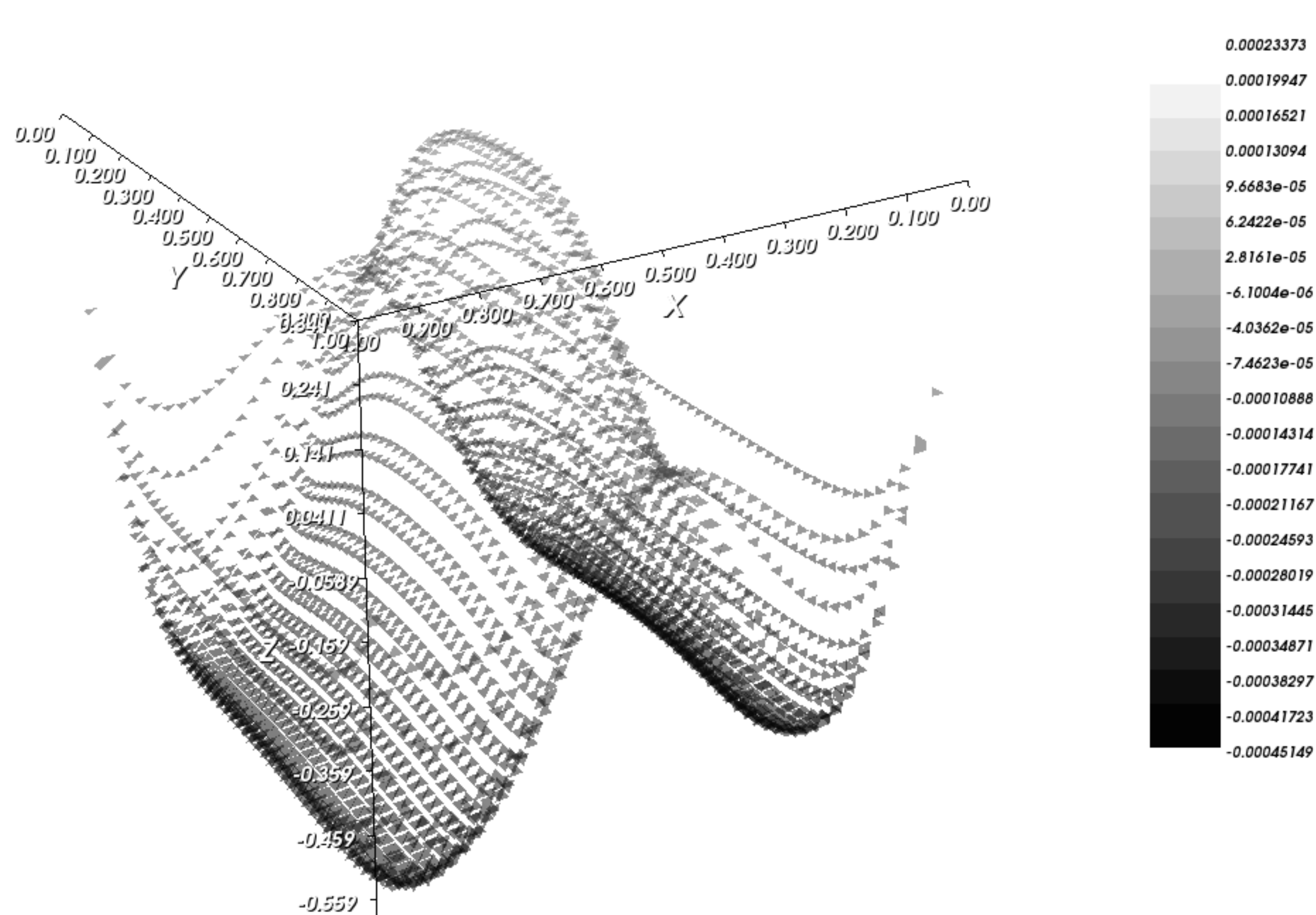}
\caption{Graph of the first product $au$.}
\label{certa}       
\end{figure}

\begin{figure}[b]
\includegraphics[scale=0.29]{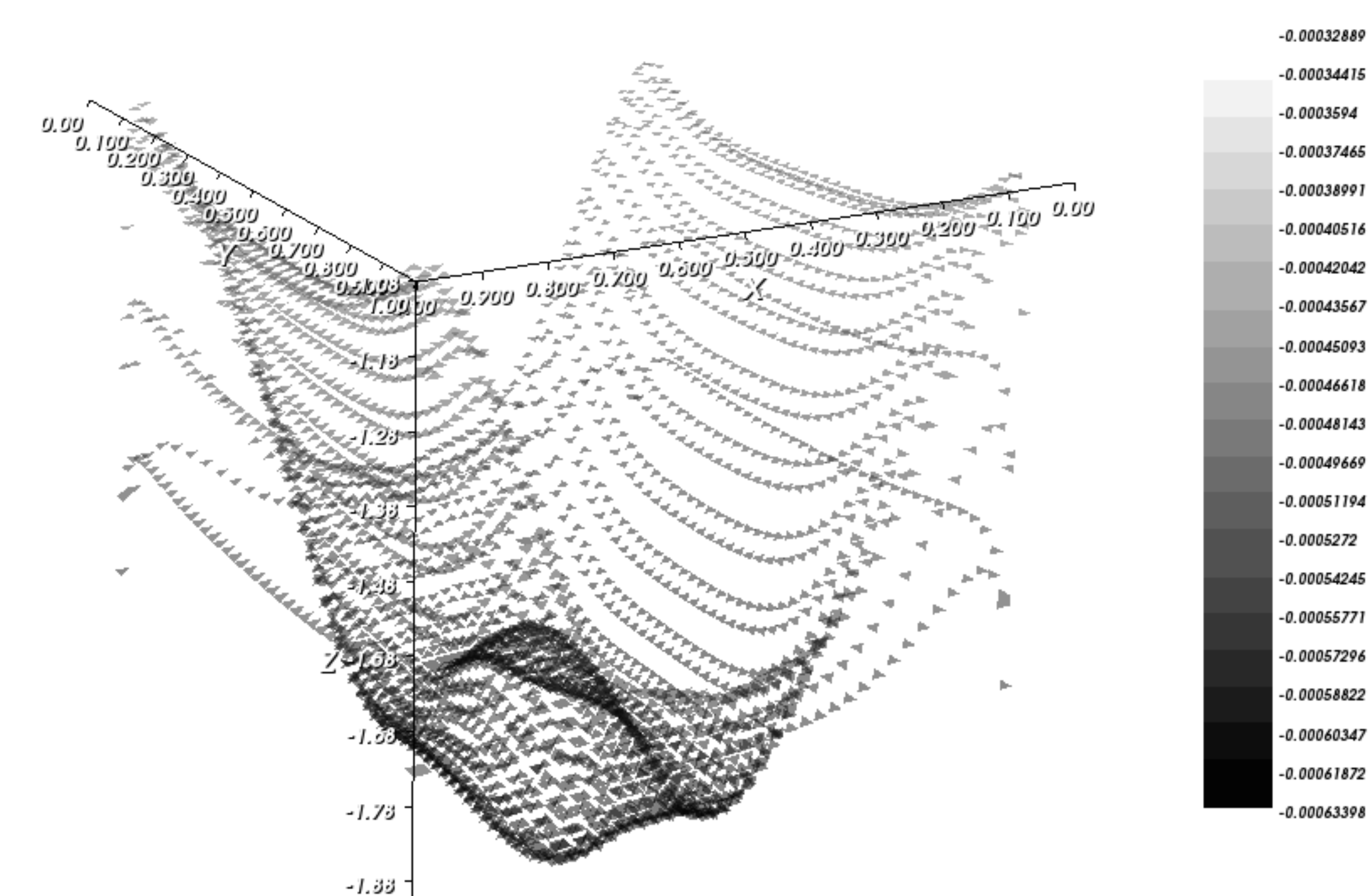}
\caption{The second product $b^-(v_--v)$.}
\label{certb-}       
\end{figure}

\begin{figure}[b]
\includegraphics[scale=0.29]{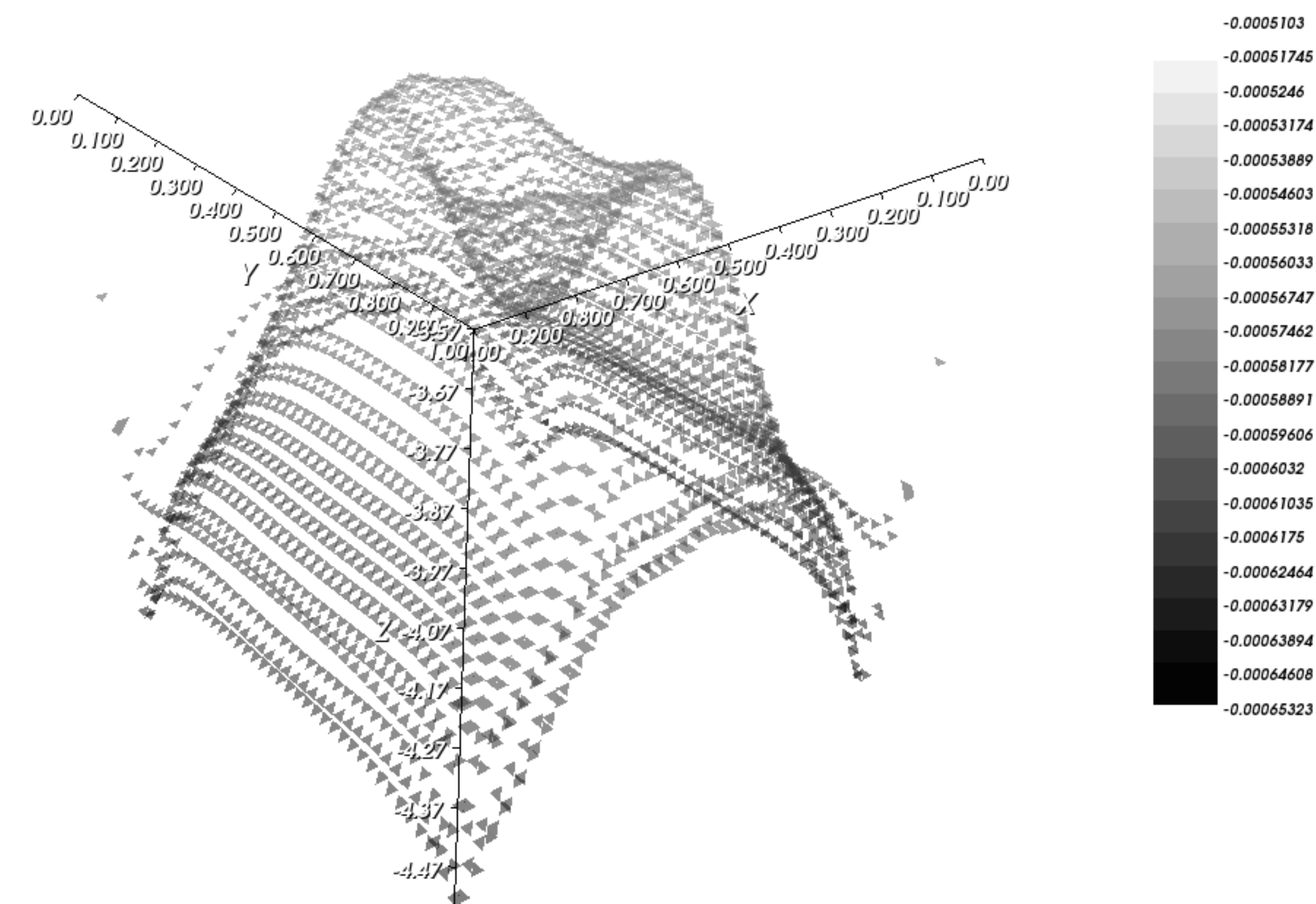}
\caption{The third product $b^+(v-v_+)$.}
\label{certb+}       
\end{figure}

\section{Asymptotic behavior as $\lambda\to\infty$}
We would like to understand the limit behavior of the optimal solutions $u_\lambda$ of our two model problems in Sections \ref{dos} and \ref{cuatro} as $\lambda\to\infty$, and in particular check that they converge to the optimal solution of the underlying optimal control problem without involving the auxiliary residual function $w$. 

For the first situation, we go back to 
$$
\hbox{Minimize in }(u, v):\quad \int_\Omega\left(\frac12|u(\bx)-\overline u(\bx)|^2+\frac\mu2|v(\bx)|^2+\frac{\lambda}2|\nabla w(\bx)|^2\right)\,d\bx
$$
subject to $(u, v)\in H^1_0(\Omega)\times L^2(\Omega)$, and 
$$
-\dv[\nabla u(\bx)+\nabla w(\bx)]+\phi(u(\bx))=v(\bx)\hbox{ in }\Omega,\quad
w=0\hbox{ on }\partial\Omega.
$$
According to our Theorem \ref{existencia} there is always an optimal pair 
$$
(u_\lambda, v_\lambda)\in H^1_0(\Omega)\times L^2(\Omega). 
$$
On the other hand, it is also well-known that the associated problem 
$$
\hbox{Minimize in }v\in L^2(\Omega):\quad 
\int_\Omega\left(\frac12|u(\bx)-\overline u(\bx)|^2+\frac\mu2|v(\bx)|^2\right)\,d\bx
$$
where
\begin{equation}\label{limite}
-\dv(\nabla u)+\phi(u)=v\hbox{ in }\Omega,\quad u=0\hbox{ in }\partial\Omega.
\end{equation}
also admits an optimal pair $(\tilde u, \tilde v)\in H^1_0(\Omega)\times L^2(\Omega)$. We would like to check how the convergence $(u_\lambda, v_\lambda)\to(\tilde u, \tilde v)$ takes place. To this end, we need to strengthen the conditions  on the non-linear term $\phi(u)$: in addition to have growth of order $N/N-2$ at infinity, as before, we demand that $\phi$ is non-decreasing so that
\begin{equation}\label{monotone}
(\phi(u_1)-\phi(u_2))(u_1-u_2)\ge0
 \end{equation}
 for every pair of numbers $u_1$, $u_2$. 

\begin{theorem}
Under the assumptions just explained, let $(u_\lambda, v_\lambda)$ be an optimal pair for its corresponding optimal control problem. There is a subsequence, as $\lambda\to\infty$, and an optimal pair $(\tilde u, \tilde v)$ such that 
$$
u_\lambda\to\tilde u\hbox{ in }H^1_0(\Omega),\quad
v_\lambda\to\tilde v\hbox{ in }L^2(\Omega).
$$
\end{theorem}
\begin{proof}
If we let
$$
I_\lambda[u, v]=\int_\Omega\left(\frac12|u(\bx)-\overline u(\bx)|^2+\frac\mu2|v(\bx)|^2+\frac\lambda2|\nabla w(\bx)|^2\right)\,d\bx,
$$
because the optimal pair $(\tilde u, \tilde v)$ is feasible for the problem depending on $\lambda$, we conclude that
\begin{equation}\label{minimo}
I_\lambda[u_\lambda, v_\lambda]\le I_\lambda[\tilde u, \tilde v]=
\int_\Omega\left(\frac12|\tilde u(\bx)-\overline u(\bx)|^2+\frac\mu2|\tilde v(\bx)|^2\right)\,d\bx,
\end{equation}
and this last integral is a constant $C$ independent of $\lambda$. In particular, $\{(u_\lambda, v_\lambda)\}$ is a uniformly bounded set of $L^2(\Omega)^2$. 
If we put $w_\lambda$ for the solution of the problem
\begin{equation}\label{landa}
-\dv[\nabla u_\lambda+\nabla w_\lambda]+\phi(u_\lambda)=v_\lambda\hbox{ in }\Omega,\quad
w_\lambda=0\hbox{ on }\partial\Omega,
\end{equation}
then it is clear that
$$
\frac\lambda2\int_\Omega|\nabla w_\lambda|^2\,d\bx\le I_\lambda[u_\lambda, v_\lambda]\le C.
$$
As $\lambda$ grows indefinitely, this inequality implies that the full sequence $w_\lambda\to0$ (strong) in $H^1_0(\Omega)$. Arguing as in the proof of Theorem \ref{existencia}, we also conclude that $\{u_\lambda\}$ is in fact uniformly bounded in $H^1_0(\Omega)$. 

Suppose that for a certain subsequence (which we do not care to relabel), we have $u_\lambda\rightharpoonup u$ in $H^1_0(\Omega)$, and $v_\lambda\rightharpoonup v$ in $L^2(\Omega)$ for certain functions $u$, $v$.  By using the weak formulation of the equations for $u_\lambda$, $v_\lambda$, and the convergence $w_\lambda\to0$, it is elementary to check that $(\tilde u, \tilde v)$ is feasible (they comply with \eqref{limite}) for the limit problem. We would like to prove that in fact $u_\lambda\to\tilde u$ (strong) in $H^1_0(\Omega)$, $v_\lambda\to\tilde v$ (strong) in $L^2(\Omega)$, and so $(\tilde u, \tilde v)$ is optimal for the limit problem. 

By subtracting the two equations \eqref{landa}, and \eqref{limite} for $(\tilde u, \tilde v)$, it is immediate to have
$$
\int_\Omega\left[(\nabla u_\lambda-\nabla\tilde u)\cdot\nabla z+\nabla w_\lambda\cdot\nabla z+(\phi(u_\lambda)-\phi(\tilde u)) z+(\tilde v-v_\lambda)z\right]\,d\bx=0
$$
for every $z\in H^1_0(\Omega)$. For $z=u_\lambda-\tilde u$, we would have
\begin{align}
\int_\Omega&\left[(\nabla u_\lambda-\nabla\tilde u)\cdot(\nabla u_\lambda-\nabla\tilde u)+\nabla w_\lambda\cdot(\nabla u_\lambda-\nabla\tilde u)\right.\nonumber\\
&\left.+(\phi(u_\lambda)-\phi(\tilde u)) (u_\lambda-\tilde u)+(\tilde v-v_\lambda) (u_\lambda-\tilde u)\right]\,d\bx=0.\nonumber
\end{align}
As far as the non-linear term
$$
\int_\Omega(\phi(u_\lambda)-\phi(\tilde u)) (u_\lambda-\tilde u)\,d\bx
$$
is concerned, we can say that it is nonnegative due to \eqref{monotone}. Notice that it does not suffice to know that 
$u_\lambda\to\tilde u$ in $L^2(\Omega)$. 
Thus
$$
\int_\Omega\left[(\nabla u_\lambda-\nabla\tilde u)\cdot(\nabla u_\lambda-\nabla\tilde u)+\nabla w_\lambda\cdot(\nabla u_\lambda-\nabla\tilde u)+(\tilde v-v_\lambda) (u_\lambda-\tilde u)\right]\,d\bx\le0.
$$
The last two terms in this inequality converge to zero because $u_\lambda\to\tilde u$ in $L^2(\Omega)$, $u_\lambda\rightharpoonup\tilde u$ in $H^1_0(\Omega)$, $v_\lambda\rightharpoonup\tilde v$ in $L^2(\Omega)$, and $w_\lambda\to0$ in $H^1_0(\Omega)$. Hence, we ought to have $u_\lambda\to\tilde u$ in $H^1_0(\Omega)$ as well. Next, by weak lower semicontinuity
\begin{equation}\label{wls}
\int_\Omega|\tilde v|^2\,d\bx\le\liminf_{\lambda\to\infty}\int_\Omega|v_\lambda|^2\,d\bx.
\end{equation}
But then, taking \eqref{minimo} into account, 
\begin{align}
\liminf_{\lambda\to\infty}I_\lambda[u_\lambda, v_\lambda]\le &
\int_\Omega\left(\frac12|\tilde u-\overline u|^2+\frac\mu2|\tilde v|^2\right)\,d\bx\nonumber\\
&\le\liminf_{\lambda\to\infty}\int_\Omega\left(\frac12|u_\lambda-\overline u|^2+\frac\mu2|v_\lambda|^2+\frac\lambda2|\nabla w_\lambda|^2\right)\,d\bx\nonumber\\
&=\liminf_{\lambda\to\infty}I_\lambda[u_\lambda, v_\lambda].\nonumber
\end{align}
We hence can conclude that  \eqref{wls} is an equality, and then $v_\lambda\to\tilde v$ in $L^2(\Omega)$. In addition
$$
\lambda\|w_\lambda\|^2_{H^1_0(\Omega)}\to0
$$
as $\lambda\to\infty$. All these strong convergences imply that the pair $(\tilde u, \tilde v)$ is optimal for the limit problem.
\end{proof}

The situation for the problem under explicit pointwise constraints for state is very much the same. Condition \eqref{monotone} can be weaken a bit in the spirit of Theorem \ref{existencia}.

\end{document}